\documentclass[smallextended]{svjour3}
\makeatletter
\let\cl@chapter\relax
\makeatother
\smartqed

\usepackage{braket,amsfonts,amsmath,amssymb,amsopn,lmodern}

\usepackage{graphicx}
\usepackage{float}                            
\usepackage[caption=false,font=small]{subfig} 

\usepackage{array,makecell,booktabs,longtable}

\usepackage{pdflscape}

\usepackage{algorithm}
\usepackage{algorithmic}

\usepackage{xcolor}

\usepackage{cite}
\usepackage{hyperref}
\usepackage{cleveref}

\newcommand{\Hess}{\nabla^2}  
\renewcommand{\Re}{\mathbb{R}}

\newcommand{\abs}[1]{\left|#1\right|}
\newcommand{\norm}[1]{\left\|#1\right\|}
\renewcommand{\Set}[1]{\left\{#1\right\}}

\newcommand{\prox}{\operatorname{prox}}
\newcommand{\argmin}{\operatorname*{argmin}}
\newcommand{\bigO}{\mathcal O}
\newcommand{\TT}{{\scriptscriptstyle TT}}
\newcommand{\MC}{{\scriptscriptstyle MC}}

\newcommand{\Z}{\mathcal Z}
\newcommand{\K}{\mathcal K}
\newcommand{\Zd}{{\scriptscriptstyle \mathcal Z}}

\newcommand{\pr}[1]{\mathbb P\left( #1 \right)}
\newcommand{\expect}[1]{\mathbb E\left[#1\right]}
\newcommand{\Var}[1]{\textrm{Var} \left[#1\right]}
\newcommand{\Cov}[1]{\textrm{Cov} \left[#1\right]}
\newcommand{\xstar}{x^\ast}
\newcommand{\Grad}{\nabla\!}

\spnewtheorem{assumption}{Assumption}{\bfseries}{\itshape}

\begin{document}

\title{Inexact Proximal Point Algorithms for Zeroth-Order Global Optimization}
\author{%
  \mbox{Minxin Zhang}
  \and \mbox{Fuqun Han}
  \and \mbox{Yat Tin Chow}
  \and \mbox{Stanley Osher}
  \and \mbox{Hayden Schaeffer}
}

\institute{
Minxin Zhang \and Fuqun Han  \and Stanley Osher \and Hayden Schaeffer \at
Department of Mathematics, University of California, Los Angeles, Los Angeles, CA 90024, USA \\
\email{\mbox{minxinzhang@math.ucla.edu}, \mbox{fqhan@math.ucla.edu}, \mbox{sjo@math.ucla.edu}, \mbox{hayden@math.ucla.edu}}
\and
Yat Tin Chow \at
Department of Mathematics, University of California, Riverside, Riverside, CA 92521, USA\\
\email{\mbox{yattinc@ucr.edu}}
}

\date{Received: date / Accepted: date}
\maketitle

\begin{abstract}
This work concerns the zeroth-order global minimization of continuous nonconvex functions with a global minimizer and possibly multiple local minimizers. We formulate a theoretical framework for inexact proximal point (IPP) methods for global optimization, establishing convergence guarantees under mild assumptions when either deterministic or stochastic estimates of proximal operators are used.
The quadratic regularization in the proximal operator and the scaling effect of a parameter $\delta>0$ create a concentrated landscape of an associated Gibbs measure that is practically effective for sampling. 
The convergence of the expectation under the Gibbs measure as $\delta\to 0^+$ is established, and the convergence rate of $\bigO(\delta)$ is derived under additional assumptions.
These results provide a theoretical foundation for evaluating proximal operators inexactly using sampling-based methods.
In particular, we propose a new approach based on tensor train (TT) approximation. 
This approach employs a randomized TT cross algorithm 
to efficiently construct a low-rank TT approximation of a discretized function using a small number of function evaluations, and we provide an error analysis for the TT-based estimation. We then propose two practical IPP algorithms, TT-IPP and MC-IPP. The TT-IPP algorithm leverages TT estimates of the proximal operators, while the MC-IPP algorithm employs Monte Carlo (MC) integration to estimate the proximal operators.
Both algorithms are designed to adaptively balance efficiency and accuracy in inexact evaluations of proximal operators. The effectiveness of the two algorithms is demonstrated through experiments on diverse benchmark functions and various applications. 
\end{abstract}


\section{Introduction}
Global optimization of nonconvex functions plays a crucial role in various scientific and engineering applications, including machine learning \cite{lecun2015deep},
signal processing \cite{saab2010sparse}, computational biology \cite{reali2017optimization} and computational physics \cite{gottvald1992global}. These problems are inherently challenging due to the presence of multiple local
minimizers and the lack of gradient information in certain scenarios. 
Standard gradient-based optimization methods, such as gradient descent, often only guarantee the convergence to a local minimizer. 
To address this, various global optimization techniques have been proposed, most of which are heuristic or require computational complexity that increases exponentially with problem dimensionality \cite{locatelli2013global}.
Zeroth-order optimization methods \cite{larson2019derivative}, also known as derivative-free optimization, solve problems solely through function evaluations, 
making them ideal for scenarios where gradient information is unavailable or expensive to compute.
In this work, we propose new \emph{inexact proximal point algorithms} for zeroth-order global minimization 
of continuous nonconvex functions $f:\Re^d\to\Re$ with a global minimizer and possibly multiple local minimizers. Theoretical convergence guarantees are established under mild assumptions.
 
Proximal point methods \cite{parikh2014proximal} is a class of optimization methods that iterate by evaluating the set-valued \emph{proximal operator},
defined as
\begin{equation}\label{eq:prox}
\prox_{tf}(x) := \argmin_{z\in\Re^d} \phi(z)\,, ~\textrm{ with } \, \phi(z) = f(z) + \frac{1}{2t} \norm{z-x}^2,
\end{equation}
for some $t>0.$ These methods are generally applied to functions for which proximal operators are either easy to compute or admit closed-form solutions.
Convergence properties of proximal point methods have been studied extensively in the context of convex optimization \cite{moreau1965proximite, rockafellar1976monotone, solodov2001unified, 
bertsekas2011incremental, asi2019stochastic}. 
For nonconvex functions, variants of proximal point methods have been considered \cite{rockafellar2021advances, davis2022proximal, fukushima1981generalized, 
khanh2023inexact,kong2019complexity,liang2018doubly}, 
typically guaranteeing convergence to a critical point
or a local minimizer.
A proximal point method for global optimization is proposed in \cite{heaton2024global},
with convergence guarantee to the global minima under the condition that the proximal operator is evaluated exactly at each iteration.
For a general nonconvex function $f$, evaluating the exact proximal operator is computationally impractical.
As a generalization of \cite{heaton2024global}, in Section~\ref{sec:IPP} we formulate a theoretical framework for inexact proximal point (IPP) methods 
that guarantees convergence to the global minimum when either deterministic or stochastic 
estimates of proximal operators are used.

We then consider zeroth-order methods for evaluating single-valued proximal operators inexactly.
The proximal operator \eqref{eq:prox} is single-valued under a
wide range of conditions (see Proposition~\ref{prop:single_prox}), encompassing a broad class of nonconvex and nonsmooth functions.
For a small $\delta>0$, it is well known that the \emph{Gibbs measure} associated with $\phi$ in \eqref{eq:prox}, defined by
\begin{equation}\label{eq:gibbs}
\rho_\delta(A):=\frac{\int_{A}\exp{(-\phi(z)/\delta)dz}}{\int_{\Re^d}\exp{(-\phi(z)/\delta)dz}} \quad\textrm{ for } A\in\mathcal B(\Re^d)\,,
\end{equation}
approximates the Dirac measure centered at $z^*:=\prox_{tf}(x)$.
The convergence in distribution of Gibbs measures and the corresponding convergence rates were derived in \cite{Gibbs_asym, bras2022convergence}. 
Let $Z_\delta$ be a random variable with the probability distribution $\rho_\delta$. Then the expectation of $Z_\delta$ satisfies
\begin{equation}\label{eq:expect}
\expect{Z_\delta}= \frac{\int z\exp\left(-\phi(z)/\delta\right)dz}{\int\exp{(-\phi(z)/\delta)dz}}\approx z^*.
\end{equation}
In Section~\ref{sec:analysis}, we show that the expectation converges to $z^*$ as $\delta\to 0^+$ if $\phi$ is continuous at around $z^*$ and
derive the convergence rate of $\bigO(\delta)$ for the case where $z^*$ is \emph{nondegenerate} and 
$\phi$ is twice continuously differentiable at around $z^*$.

To obtain an estimate of $z^*$, it is impractical to directly compute \eqref{eq:expect} via numerical integration due to the exponential growth in the number of quadrature nodes with respect to the dimension $d$. Fortunately, the quadratic regularization in $\phi$ and the scaling effect of $\delta$ lead to a concentrated landscape of the Gibbs measure that is practically effective for sampling.
As an illustration, Figure~\ref{fig:regula} compares the original landscape of a nonconvex function, the Schaffer function \cite{test_problems_2005},
with the landscapes of its several transformations. The minimizer of the original function $f$ in Figure~\ref{fig:1a} is turned into a maximizer in Figure~\ref{fig:1b}; the quadratic regularization reduces oscillations and increases density near the solution in Figure~\ref{fig:1c}; and
the scaling by a small $\delta>0$ concentrates the density near the solution.
We consider two sampling-based methods to efficiently estimate the proximal operator. A classical approach, as proposed in 
\cite{osher2023hamilton,heaton2024global,tibshirani2024laplace}, 
is to use the Monte Carlo (MC) integration to approximate \eqref{eq:expect} by computing a weighted average of Gaussian samples 
centered at $x$. The MC integration is easy to implement, but 
may suffer from high variance in practice. To alleviate this, variance reduction techniques, such as the exponentially 
weighted moving average (EWMA) \cite{ross2009probability}, can be incorporated.
Additionally, in Section~\ref{sec:approx}, we propose a new approach based on tensor train (TT) approximation \cite{oseledets2011tensor,Thm_TT_accuracy},
which exploits the Sobolev smoothness of the integrands to improve the estimation accuracy. 
The TT approximation of tensors is a generalization of the truncated singular value decomposition (SVD) of matrices.
To obtain a TT estimate of \eqref{eq:expect}, we employ the randomized TT cross algorithm 
\cite{oseledets2010tt, savostyanov2014quasioptimality} to
construct a low-rank TT approximation of the discretized function $\exp{\left(-f/\delta\right)}$ over a mesh grid. 
The TT cross algorithm accurately computes the TT approximation using a small number of function evaluations, with computational cost depending linearly on the dimension $d$, and without storing the full tensor.
This makes it particularly well-suited for functions with approximate low-rank structure, such as the one illustrated in Figure~\ref{fig:1d}.
We provide an error analysis for the TT estimate of a proximal operator in Section~\ref{sec:error}.

\begin{figure}[H]
  \centering
  \subfloat[]{%
    \includegraphics[width=0.24\textwidth]{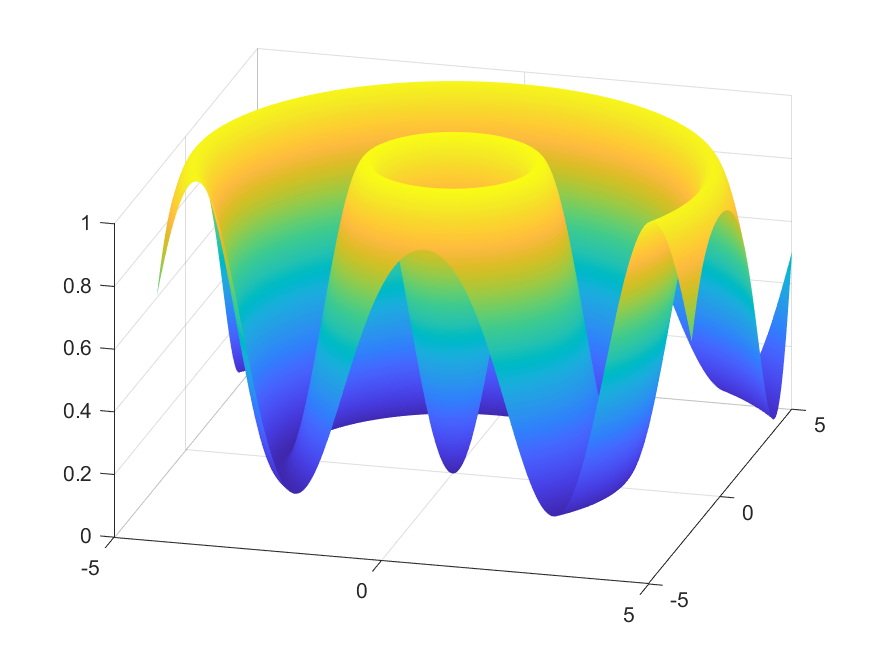}%
    \label{fig:1a}%
  }\hfill
  \subfloat[]{%
    \includegraphics[width=0.24\textwidth]{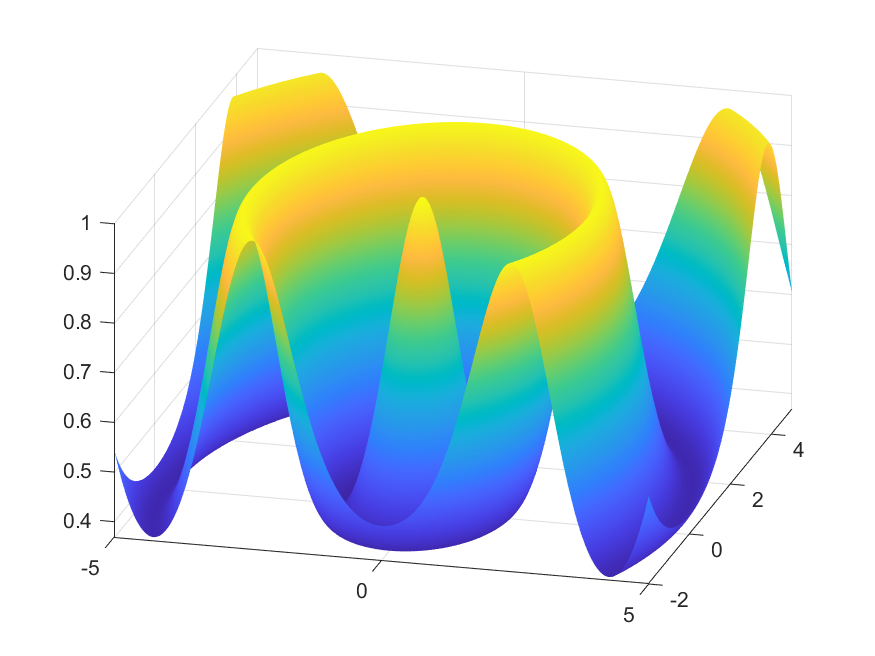}%
    \label{fig:1b}%
  }\hfill
  \subfloat[]{%
    \includegraphics[width=0.24\textwidth]{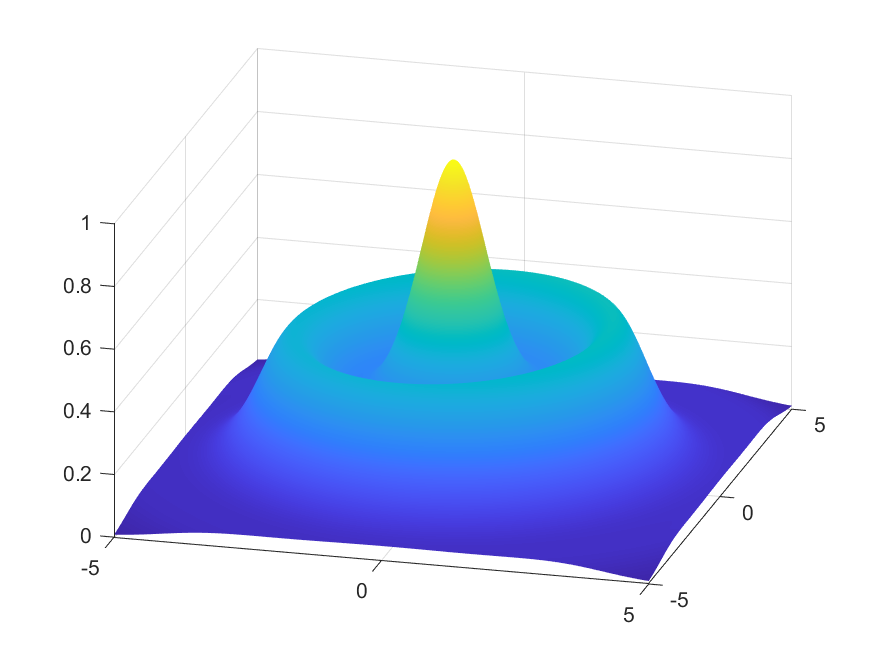}%
    \label{fig:1c}%
  }\hfill
  \subfloat[]{%
    \includegraphics[width=0.24\textwidth]{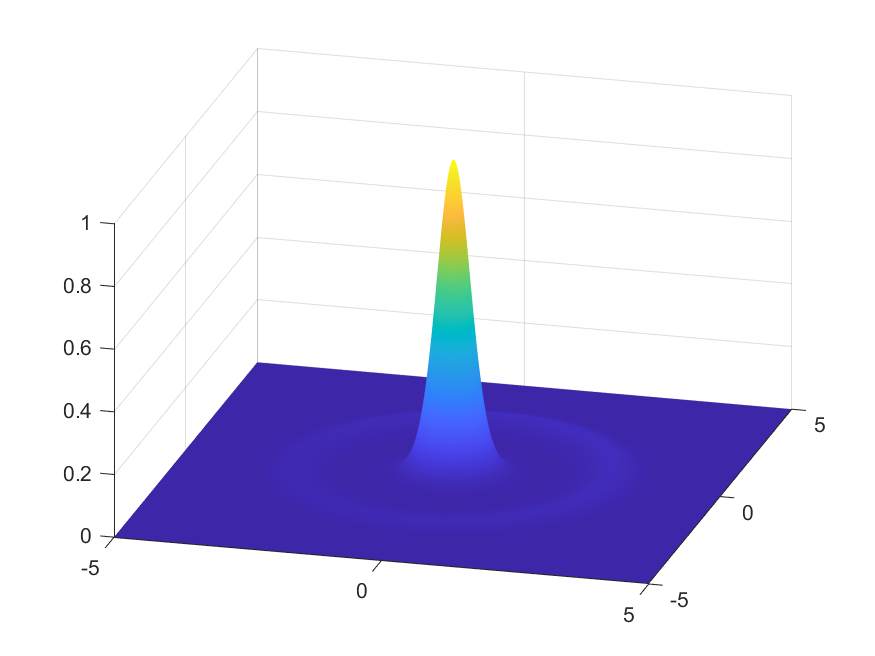}%
    \label{fig:1d}%
  }

  \caption{%
    (a) The 2D Schaffer function $f(x)$; 
    (b) $\exp(-f(x))$; 
    (c) $\exp\!\bigl(-\bigl(f(x)+\|x-z\|_2^2/(2t)\bigr)\bigr)$ with $t=6$, $z=(1,1)$; 
    (d) $\exp\!\bigl(-\bigl(f(x)+\|x-z\|_2^2/(2t)\bigr)/\delta\bigr)$ with $\delta=0.25$.%
  }
  \label{fig:regula}
\end{figure}

We then propose two practical IPP algorithms: the TT-IPP algorithm, which leverages TT estimates of the proximal operators; 
and the MC-IPP algorithm, which employs MC integration to estimate the proximal operators. Both algorithms adaptively decrease the parameter $\delta$
based on whether a sufficient decrease in the function value is achieved relative to several previous iterates.
To balance efficiency and accuracy in estimating the proximal operators, TT-IPP is designed to adaptively refine the mesh grid and update the
associated TT approximation; and MC-IPP is designed to adaptively increase the sample size and update the EWMA parameter.
Additionally, a warm start can be conveniently incorporated into the two IPP algorithms, with minimal computational cost equivalent to a single iteration of either TT-IPP or MC-IPP.
The effectiveness of both algorithms is shown through experiments on diverse benchmark functions and applications.

\subsection{Prior work}\label{sec:prior}
In this work, we focus on zeroth-order global optimization for nonconvex functions. Numerous optimization methods in the literature fall into this category. 
A comprehensive survey of historical perspectives and recent advancements in global optimization is presented in \cite{locatelli2021global}.
For example, pure random search (PRS) \cite{locatelli2013global} samples random points uniformly over the feasible region and selects the one with the smallest function value as the solution. 
Genetic Algorithms (GAs) \cite{holland1992genetic} start with a population of candidate solutions and evolve them by mimicking natural evolutionary processes.
Differential evolution (DE) \cite{DE_original} iteratively refines candidate solutions by combining differences between randomly selected agents 
to explore the search space.
Particle swarm optimization (PSO) \cite{PSO_first} is inspired by the social behavior of swarms, 
where particles explore the search space by updating their positions based on personal and collective best solutions.
Simulated Annealing (SA) \cite{sa_original} is a probabilistic algorithm inspired by the annealing process in metallurgy, 
which explores the search space by accepting both improving and, with decreasing probability, worsening solutions to escape local optima.
Most of the aforementioned methods are metaheuristic.
Additionally, variants of random zeroth-order methods for convex optimization were proposed in \cite{nesterov2017random}, 
where an expectation similar to \eqref{eq:expect} was used to approximate gradients. 
More zeroth-order methods for convex optimization are described in \cite{nemirovskij1983problem}.
\cite{jongeneel2024small} proposed a novel randomized gradient estimator for zeroth-order optimization of real analytic functions.
Random zeroth-order methods for constrained minimization of subdifferentially polynomially bounded functions 
were proposed in \cite{lei2024subdifferentially}.
Consensus-based optimization (CBO) \cite{CBO_first_2017, CBO_Analysis} is an emerging class of zeroth-order methods for global optimization that offers convergence 
guarantees to near-optimal solutions. In CBO, a swarm of agents collectively moves toward a consensus point while using stochastic 
perturbations to explore the search space.
In \cite{gomes2023derivativefree}, a derivative-free global optimization algorithm was introduced for one-dimensional functions, 
providing certain convergence guarantees. The algorithm approximates gradient flow using MC integration and rejection sampling.
A proximal point algorithm called HJ-MAD was proposed in \cite{heaton2024global} for global optimization of nonconvex functions. The method iteratively 
computes the proximal operator via MC integration; however, convergence is guaranteed only if the proximal operator is evaluated exactly at each iteration.
In \cite{engquist2024adaptive}, a stochastic derivative-free algorithm was proposed, whose continuous limit, modeled by a stochastic differential equation, converges to the global minimizer as time approaches infinity.

In recent years, tensor-train-based methods have gained attention for multidimensional optimization. 
TT-Opt \cite{sozykin2022ttopt} employs TT approximations to perform optimization on a predefined grid, which limits its ability to achieve high accuracy or explore large domains. PROTES \cite{batsheva2024protes} is another method for optimization on a predefined grid, utilizing probabilistic sampling from a low-parametric distribution represented in a TT format. In 
\cite{chertkov2022optimization}, a probabilistic algorithm was proposed for optimizing discretized functions in the TT format. 
TTGO \cite{shetty2023tensor} leverages the separable structure of TT approximations to perform conditional sampling for initializing local optimization solvers in robotics applications. In \cite{soley2021iterative}, an iterative power algorithm was proposed for global optimization, which performs power iterations on a special form of TT approximations, the quantics tensor train, to concentrate the density function near the global minimizer.

\subsection{Contributions and organization}
The contributions of this work are summarized below.
\begin{enumerate}
    \item  A theoretical framework for IPP methods is formulated for the global optimization of nonconvex functions, with convergence guarantees established under mild assumptions when either deterministic or stochastic estimates of proximal operators are used. Unlike previous analyses in the literature on IPP methods for 
    nonconvex functions, which focus on convergence guarantees to a stationary point, we establish convergence to a global minimizer.
    \item  Convergence of the expectation \eqref{eq:expect} under Gibbs measure as $\delta\to 0^+$ is established, and the convergence
    rate of $\bigO(\delta)$ is derived under additional assumptions.
    These results provide theoretical foundations for evaluating proximal operators inexactly using sampling-based methods.
    \item A TT-based approach is proposed for the estimation of proximal operators, accompanied by an error analysis.
    This approach leverages the Sobolev smoothness of functions to circumvent the \emph{curse-of-dimensionality}, a challenge faced by most existing global optimization methods.
    \item Building on the theoretical framework, two practical IPP algorithms, TT-IPP and MC-IPP, are developed. 
    These algorithms are designed to adaptively balance efficiency and accuracy in evaluating inexact proximal operators. 
    Their effectiveness is demonstrated through experiments on a diverse set of benchmark functions and various applications.
\end{enumerate}

The rest of the paper is organized as follows. The theoretical framework of IPP methods for global optimization is formulated
and analyzed in Section~\ref{sec:IPP}. Convergence results of the expectation under the parameterized Gibbs measure are established 
in Section~\ref{sec:analysis}. Section~\ref{sec:approx} introduces the new TT-based approach for inexact evaluation of proximal operators, along with an error analysis. 
The two practical algorithms, TT-IPP and MC-IPP, are proposed in Section~\ref{sec:IPP_alg}.
Experiment results on benchmark functions and practical applications are showcased in Section~\ref{sec:experiments}. Finally, Section~\ref{sec:conclud} concludes the paper, discussing limitations of this work and potential directions for future research.

\section{Inexact proximal point methods for global optimization}\label{sec:IPP}
In this section, we provide a theoretical framework of inexact proximal point (IPP) methods for the global optimization of nonconvex functions. 
An IPP method evaluates the proximal operator inexactly
\[
y^k\approx \hat x^k\in\prox_{t_k f}(x^k) 
\]
at each iterate $x^k$ for some $t_k>0,$ and the next iterate is given by
\[
x^{k+1}=\alpha_k y^k+(1-\alpha_k) x^k\,,
\]
where $\alpha_k\in (0,1]$ is called a \emph{damping} parameter.
Details of the IPP method under consideration are summarized in Algorithm~\ref{alg:IPP}. In particular, Lines~\ref{line:qk}--\ref{line:tk} update $t_k$ 
in the same manner as the exact proximal point method described in \cite{heaton2024global} to prevent the iterates
from converging to a local minimizer that is not globally optimal. Specifically, when $x^{k+1}$ is close to the previous iterate
$x^k,$ $t_k$ is increased to encourage global exploration; when $x^{k+1}$ is farther from $x^k$, $t_k$ is decreased to promote
exploration near the current iterate $x^{k+1}$; otherwise, $t_k$ remains unchanged.

\begin{algorithm}[htp!]
\caption{Inexact Proximal Point Method (IPP)}\label{alg:IPP}
\begin{algorithmic}[1] 
    \STATE \textbf{Input:} {$x^0\in\Re^d$, $0<\eta_-<1<\eta_+$, $0<\theta_1\le\theta_2<1$, $\bar\epsilon>0$,
    $0< \tau\le t_0\le T$, $\Set{\alpha_k}\subset [\alpha_{\min},\alpha_{\max}]\subset (0,1]$. }
    \FOR{$k=0,1,2,\cdots$}
        \STATE $y^k\approx \hat x^k\in\prox_{t_k f}(x^k)$
        \STATE $x^{k+1}=\alpha_k y^k+(1-\alpha_k) x^k$
        \STATE $q_{k} = \norm{x^{k+1}-x^k}/t_k$ \label{line:qk}
        \IF{$k\ge 1$ and $q_k\le\theta_1 q_{k-1}+\bar\epsilon$} \label{line:qk1}
            \STATE $t_{k+1} = \min\{\eta_+t_k, T\}$
        \ELSIF{$k\ge 1$ and $q_k>\theta_2 q_{k-1}+\bar\epsilon$}
            \STATE $t_{k+1} = \max\{\eta_-t_k, \tau\}$
        \ELSE
            \STATE $t_{k+1} = t_k$
        \ENDIF \label{line:tk}
    \ENDFOR
    \STATE \textbf{Output:} last iterate $x^{k}$.
\end{algorithmic}
\end{algorithm}


We include the definition of the subdifferential below. 
\begin{definition}\cite[Section 2]{davis2018subgradient}\label{def:subd}
 The \emph{subdifferential} of $f:\Re^d\to\Re$ at $x$, denoted by $\partial f(x),$ is the set of all $v\in\Re^n$ satisfying
 \[
 f(y)\ge f(x)+\langle v,y-x\rangle+o(\norm{y-x}) ~\textrm{ as }~ y\to x\,.
 \]
\end{definition}
To establish the theoretical convergence of IPP methods, we make the following assumptions on $f$.
\begin{assumption}\label{assum1}
The function $f:\Re^d\to\Re$ is continuous and has a global minimizer $x^*\in\Re^d$ such that $f_{\min}:=f(x^*).$
\end{assumption}
\begin{assumption}\label{assum2}
The function $f:\Re^d\to\Re$ is $p$-coercive for some $p>0,$ i.e.
\[
f(x)/\norm{x}^p\to\infty \textrm{ as } \norm{x}\to\infty\,.
\]
\end{assumption}
\begin{assumption}\label{assum3}
There exists $\mu>0$ such that $0\in\partial f(x)$ and $f(x)<f_{\min}+\mu$ imply $f(x)=f_{\min}.$
\end{assumption}
We remark that Assumptions~\ref{assum1} and \ref{assum2} are standard conditions ensuring the continuity of $f$ and the existence of
a global minimizer, and Assumption~\ref{assum3} is a mild condition that excludes cases where 
$f$ exhibits extreme oscillations around the global minimizer.

\begin{lemma}
Under the Assumptions~\ref{assum1}--\ref{assum3}, for arbitrary $t>0$,
$\prox_{tf}(x)$ is nonempty for all $x.$ 
\end{lemma}
\begin{proof}
A similar result is stated in \cite[Lemma A1]{heaton2024global}, but under stronger assumptions, and the same proof can be applied here.
\end{proof}

The following theorem establishes the convergence of Algorithm~\ref{alg:IPP} to the global minimum
when the proximal operators
are evaluated inexactly with asymptotic accuracy.
The inexactness condition \eqref{eq:prox_error} was also
used in \cite{rockafellar1976monotone} to establish the convergence of a classical IPP method for convex optimization, and in \cite{rockafellar2021advances} to show its local convergence to a stationary point
for nonconvex functions using monotone operator theory.
To the best of our knowledge, this is the first result that guarantees convergence of an IPP method to a global minimizer. 
\begin{theorem}\label{thm:IPP}
Suppose Assumptions~\ref{assum1} -- \ref{assum3} hold, the parameter $\alpha_{\min}>1-\eta_-$, and the choice of $T>0$ is sufficiently large (see \eqref{eq:tbound}) in Algorithm~\ref{alg:IPP}.
Let $\Set{x^k}_{k\ge 0}$ be the sequence of iterates generated by Algorithm~\ref{alg:IPP}. 
If the error in estimating the proximal point operator satisfies
\begin{equation}\label{eq:prox_error}
\sum_{k=0}^{\infty} \norm{y^k - \hat x^k}^2<\infty\,,
\end{equation}
where $\hat x^k\in \prox_{t_k f}(x^k)$, then $\Set{f(x^k)}_{k\ge 0}$ converges to $f_{\min}$ as $k\to\infty.$
\end{theorem}
\begin{proof}
First, we show that $t_k = T$ for all $k$ sufficiently large. Indeed, for each $k$, since $\hat x^{k+1}\in \prox_{t_{k+1} f}(x^{k+1})$,
\[
f(\hat x^{k+1})+\frac{1}{2t_{k+1}}\norm{\hat x^{k+1}-x^{k+1}}^2
\le f(\hat x^{k})+\frac{1}{2t_{k+1}}\norm{\hat x^{k}-x^{k+1}}^2\,,
\]
which, combined with the convexity of $\norm{\cdot}^2$ and $x^{k+1} = \alpha_ky^k + (1-\alpha_k)x^k$, implies
\begin{align}\label{eq:xdiff}
& f(\hat x^{k+1})-f(\hat x^{k})\nonumber 
 \le   \frac{1}{2t_{k+1}}\left(\norm{\hat x^{k}-x^{k+1}}^2-\norm{\hat x^{k+1}-x^{k+1}}^2\right)\nonumber \\ 
= &\frac{1}{2t_{k+1}}\left(\norm{\alpha_k(\hat x^{k}- y^k)+(1-\alpha_k)(\hat x^k- x^k)}^2-\norm{\hat x^{k+1}-x^{k+1}}^2\right)\nonumber\\
\le & \frac{1}{2t_{k+1}}\left(\alpha_k \norm{\hat x^{k}- y^k}^2+(1-\alpha_k)\norm{\hat x^k- x^k}^2-\norm{\hat x^{k+1}-x^{k+1}}^2\right).
\end{align}
Therefore, combining the above inequity with the iterative parameter updates in Algorithm~\ref{alg:IPP} yields
\begin{align}\label{eq:boundfk}
-\infty<&f(x^*)\le  \limsup_{k\to\infty} f(\hat x^{k+1}) \nonumber\\ \le & f(\hat x^{0})+\sum_{k=0}^\infty
\frac{\alpha_k \norm{\hat x^{k}- y^k}^2+(1-\alpha_k)\norm{\hat x^k- x^k}^2-\norm{\hat x^{k+1}-x^{k+1}}^2}{2t_{k+1}}\nonumber\\
\le & f(\hat x^{0})+\frac{1}{2\tau}\sum_{k=0}^\infty\norm{\hat x^{k}- y^k}^2+
\frac{1-\alpha_{\min}}{\eta_-}\sum_{k=0}^\infty\frac{\norm{\hat x^k- x^k}^2}{2t_{k}}-\sum_{k=1}^\infty\frac{\norm{\hat x^k- x^k}^2}{2t_{k}}\nonumber\\
\le & f(\hat x^{0})+\frac{1}{2\tau}\sum_{k=0}^\infty\norm{\hat x^{k}- y^k}^2+
\frac{1-\alpha_{\min}-\eta_-}{2T\eta_-}\sum_{k=0}^\infty\norm{\hat x^k- x^k}^2\,.
\end{align}
By the assumptions that $\alpha_{\min}>1-\eta_-$ and \eqref{eq:prox_error}, it follows that $\sum_{k=0}^\infty\norm{\hat x^{k}-x^{k}}^2<\infty\,,$ which implies
\begin{equation}\label{eq:xk_approx}
\lim\limits_{k\to\infty} \norm{\hat x^{k}-x^{k}} = 0\,.
\end{equation}
Moreover, since $\alpha_k\in (0,1],$ by the convexity of $\norm{\cdot}$ and \eqref{eq:prox_error},
\begin{equation}\label{eq:xk1_approx}
0\le \lim_{k\to\infty}\norm{\hat x^{k}-x^{k+1}}\le\lim_{k\to\infty}\alpha_k\norm{\hat x^{k}-y_k} +(1-\alpha_k)\norm{\hat x^{k}-x^{k}}=0\,.
\end{equation}
By the triangle inequality, combining \eqref{eq:xk_approx} and \eqref{eq:xk1_approx} gives 
\[
0\le \lim_{k\to\infty}\norm{x^{k+1}-x^{k}}\le\lim_{k\to\infty}\norm{\hat x^{k}-x^{k}}+\lim_{k\to\infty}\norm{\hat x^{k}-x^{k+1}}=0\,.
\]
Therefore, $\lim\limits_{k\to\infty}\norm{x^{k+1}-x^{k}}=0.$ For $q_k$ in Line~\ref{line:qk} of Algorithm~\ref{alg:IPP}, since $t_k\in [\tau,T]\subset (0,\infty)$ for all $k,$
\[
\lim_{k\to\infty} q_k = \lim_{k\to\infty} \norm{x^{k+1}-x^k}/t_k = 0\,,
\]
which implies 
\[
\lim_{k\to\infty} \frac{q_k}{\theta_1 q_{k-1}+\bar\epsilon}= 0\,.
\]
Hence, the condition in Line~\ref{line:qk1} of Algorithm~\ref{alg:IPP} is satisfied for all sufficiently large \( k \), and \( t_k \) reaches the upper bound \( T \) within a finite number of iterations.

Next, we show that $\Set{\hat x^k}_{k\ge 0}$ and $\Set{x^k}_{k\ge 0}$ are bounded. Indeed, \eqref{eq:boundfk} implies that $\Set{f(\hat x^{k})}$ is bounded. 
If $\Set{\hat x^k}$ is unbounded, then there exists a subsequence 
$\Set{\hat x^{k_j}}$ such that $\lim_{j\to\infty}\norm{\hat x^{k_j}}=\infty,$ which contradicts
Assumption~\ref{assum2}. Therefore, $\Set{\hat x^k}$ is bounded. By \eqref{eq:prox_error}, it follows that
$\Set{y^k}_{k\ge 0}$ is bounded as well. Thus, it can be proved by induction that, there exists a constant $M>0$ such that, for all $k$,
\begin{equation}\label{eq:boundM}
\max\left\{\norm{\hat x^k}^2, \norm{x^k}^2\right\}\le M.
\end{equation}

Now we show there exists a subsequence of $\Set{f(x_k)}$ that converges to $f_{\min}.$
Since $\Set{\hat x^k}$ is bounded, by Bolzano-Weierstrass theorem, there exists a convergent subsequence, $\Set{\hat x^{k_j}}$, with
limit $x^{\infty}.$ By \eqref{eq:xk_approx},
\[
0\le \lim_{j\to\infty} \norm{x^{k_j}-x^{\infty}}\le \lim_{j\to\infty} \left(\norm{x^{k_j}-\hat x^{k_j}} +\norm{\hat x^{k_j}-x^{\infty}}\right)=0\,,
\]
i.e. $\lim\limits_{j\to\infty}x^{k_j} = x^{\infty}.$ 
Next, we show that $0\in\partial f(x^\infty).$ By the definition of Moreau envelope, we have
\begin{equation}
\label{eq:fu_1}
    f(x^\infty)\ge u(x^\infty,T):= \min_{z} f(z) + \frac{1}{2T} \norm{z-x^\infty}^2.
\end{equation}
On the other hand, for $\hat x^\infty\in \prox_{T f}(x^\infty),$
\begin{align}
f(x^\infty) = &\lim_{j\to\infty} f(\hat x^{k_j}) \le \lim_{j\to\infty}f(\hat x^{k_j}) + \frac{1}{2T}\|\hat x^{k_j}-x^{k_j}\|^2\notag\\ 
\le& \lim_{j\to\infty} f(\hat x^\infty) + \frac{1}{2T}\norm{\hat x^\infty -x^{k_j}}^2\notag\\
\le & \lim_{j\to\infty} u(x^\infty,T)+\frac{1}{2T}\left(\norm{\hat x^\infty -x^{k_j}}^2-\norm{\hat x^\infty -x^{\infty}}^2\right)\notag\\
=  &u(x^\infty,T)\,.
\label{eq:fu_2}
\end{align}
 Hence, \eqref{eq:fu_1} and \eqref{eq:fu_2} lead to
$f(x^\infty)= u(x^\infty,T),$ which,
by Definition~\ref{def:subd}, implies $0\in\partial f(x^\infty)$.
If $T>0$ is sufficiently large such that 
\begin{equation}\label{eq:tbound}
T>M/\mu
\end{equation}
for $M$ given in \eqref{eq:boundM} and $\mu$ given in Assumption~\ref{assum3}, then
\[
f(x^\infty) = u(x^\infty,T)\le f(x^*)+\frac{1}{2T}\norm{x^*-x^\infty}^2< f(x^*)+\mu\,.
\]
It follows from Assumption~\ref{assum3} that $f(x^\infty)=f_{\min}.$ 

It remains to show that the whole sequence $\Set{f(x^k)}_{k\ge 0}$ converges to $f_{\min}.$
Since $\sum\limits_{k=0}^{\infty} \norm{x^{k+1} - \hat x^k}^2<\infty$ and $\sum\limits_{k=0}^\infty\norm{\hat x^{k}-x^{k}}^2<\infty$, 
for arbitray $\kappa>0,$ there exists a constant $N_\kappa$ such that for all integers $k>l>N_\kappa,$
\[
\sum_{j=l}^{k-1} \norm{x^{j+1} - \hat x^j}^2<\kappa ~\textrm{ and } \sum_{j=l}^{k} \norm{x^{j+1} - \hat x^{j+1}}^2<\kappa \,.
\]
By \eqref{eq:xdiff}, for $k>l>N_\kappa,$
\begin{align*}
&f(\hat x^{k})-f(\hat x^{l})\\
\le & \frac{1}{2\tau}\sum_{j=l}^{k-1}\left(\alpha_j\norm{\hat x^{j}-y^j}^2+(1-\alpha_j)
\norm{\hat x^{j}-x^j}^2\right)-\frac{1}{2T}\sum_{j=l}^{k-1}\norm{\hat x^{j+1}-x^{j+1}}^2\\
<&\kappa\left(\frac{1}{2\tau}-\frac{1}{2T}\right).
\end{align*}
Since $\kappa>0$ is arbitrary, $\Set{f(\hat x^k)}_{k\ge 0}$ is convergent. Therefore,
$\lim\limits_{k\to\infty} f(\hat x^k) = \lim\limits_{j\to\infty} f(\hat x^{k_j}) = f_{\min}\,.$
By the continuity of $f$ and \eqref{eq:xk_approx},
\[
\lim\limits_{k\to\infty} f(x^k) = \lim\limits_{k\to\infty} f(\hat x^k) = f_{\min}.
\]
The proof is thus completed.
\end{proof}

The following theorem establishes the convergence of IPP to the global minimum when the iterates $\Set{x^k}_{k\ge 0}$ are based on
stochastic estimates of the proximal operators.
\begin{theorem}\label{thm:sIPP}
Suppose Assumptions~\ref{assum1} -- \ref{assum3} hold, $\alpha_{\min}>1-\eta_-$, and the choice of  $\,T>0$ is sufficiently large (see \eqref{eq:ptbound}) in Algorithm~\ref{alg:IPP}.
Let $\Set{x^k}_{k\ge 0}$ be a stochastic sequence of iterates generated by Algorithm~\ref{alg:IPP}. 
If there exists constants $\Set{\epsilon_k}$ with
$\sum\limits_{k=0}^\infty\epsilon_k<\infty$ and probabilities $\Set{p_k}$ with $\sum\limits_{k=0}^\infty p_k<\infty$ such that
\begin{equation}\label{eq:pprox_error}
\pr{\norm{y^k-\hat x^k}^2>\epsilon_k}\le p_k\,,
\end{equation}
where $\hat x^k\in \prox_{t_k f}(x^k)$, then $\Set{f(x^k)}_{k\ge 0}$ converges to $f_{\min}$ almost surely as $k\to\infty.$
\end{theorem}
\begin{proof}
See Appendix~\ref{app:proof27}.
\end{proof}

 \section{Approximation of the proximal operator}\label{sec:analysis}
 As discussed in the previous section, the convergence of an IPP method relies on estimating 
the proximal operators with sufficient asymptotic accuracy. In this section, we consider the approximation of $\operatorname{prox}_{t f}(x)$ via the Gibbs measure associated with the function
 \begin{equation}\label{eqn:phi}
 \phi(z) := f(z)+\frac{1}{2t}\norm{z-x}^2.
 \end{equation}
We focus on the case where the proximal operator is single-valued, i.e. the function $\phi$ has a unique global minimizer 
$z^*=\prox_{tf}(x).$
In this case, the Gibbs measure defined by \eqref{eq:gibbs} approximates the Dirac measure centered at $z^*$ for some small $\delta>0$.
Hence, an approximation of the proximal operator is given by 
\begin{equation}\label{eqn:aprox}
\prox_{tf}^\delta(x) := \frac{\int z\exp\left(-\phi(z)/\delta\right)dz}{\int\exp{(-\phi(z)/\delta)dz}}\approx \prox_{tf}(x)\,.
\end{equation}

The proximal operator is single-valued under a wide range of conditions. For relatively small $t>0,$ a standard result is that it is single-valued if $f$ is \emph{prox-regular}, a property that holds for all convex functions and a broad class of nonconvex functions \cite{poliquin2010calculus}.
For sufficiently large $t$, it is single-valued if $x^*$ is a \emph{nondegenerate} minimizer when $f$ is $C^2$ around $x^*$, 
or if $f$ is \emph{sharp} \cite{davis2018subgradient,dinh2017sharp} around $x^*$ in the nonsmooth case. 
These conditions and required definitions are summarized in the following definitions and Proposition~\ref{prop:single_prox} below.
\begin{definition}
A stationary point $x^*$ of $f$ is called \emph{nondegenerate} if $f$ is $C^2$ around $x^*$ and the Hessian $\Hess f(x^*)$ is nonsingular.
\end{definition}

\begin{definition}
 We say $f$ is \emph{sharp} on a neighborhood $U$ of $x^\ast$ if there exists $\eta>0$ such that
    \begin{equation}\label{eq:sharp}
    f(x)-f_{\min}\ge \eta\norm{x-x^\ast},\quad \forall x\in U\,.
    \end{equation} 
\end{definition}
The sharpness condition holds for a wide class of nonconvex nonsmooth functions, see e.g., 
\cite{candes2015phase, hoffman2003approximate, attouch2013convergence, bolte2014proximal, karimi2016linear}.

 \begin{proposition}\label{prop:single_prox}
The proximal operator $\operatorname{prox}_{t f}(x)$ is single-valued under any of the following conditions:
\begin{enumerate}  
    \item \textbf{Prox-Regularity of $f$:} If $f$ is \emph{prox-regular} at $x$, meaning there exists a constant $r > 0$ such that for all $x' \neq x''$ near $x$,
    \[
    f(x'') > f(x') + \langle v, x'' - x' \rangle - \frac{r}{2} \| x'' - x' \|^2,
    \]
    where $v \in \partial f(x')$, then $\operatorname{prox}_{t f}(x)$ is single-valued for $t < 1/r$; or
     \item \textbf{Nondegeneracy of $x^\ast$:} In addition to Assumptions~\ref{assum1} -- \ref{assum3}, if $x^\ast$ 
    is a unique global minimizer of $f$, nondegenerate, and
    if $f$ is $C^2$ on a neighborhood $U$ of $x^\ast$, then $\operatorname{prox}_{t f}(x)$ is single-valued for all sufficiently large $t$; or
    \item \textbf{Sharpness at $x^\ast$:} In addition to Assumptions~\ref{assum1} -- \ref{assum3}, 
    if $x^\ast$ is a unique global minimizer of $f$, and
    if $f$ is sharp on a neighborhood $U$ of $x^\ast,$
    then $\operatorname{prox}_{t f}(x)$ is single-valued for all sufficiently large $t$.
\end{enumerate}
\end{proposition}
\begin{proof}
See Appendix~\ref{app:proof31}.
\end{proof}

The following theorem indicates the convergence of the approximate proximal operator given in \eqref{eqn:aprox} as $\delta\to 0^+$ under the 
condition that $\phi$ is continuous around $z^*.$
A similar result was shown in \cite{tibshirani2024laplace}, but under stronger assumptions.
\begin{theorem}
\label{thm:conv}
Assume that $\phi:\mathbb R^d\to \mathbb R$ is $p$-coercive for some $p>0$, and that $\phi$ has a unique global minimizer $z^*$. If there exists a neighborhood $U$ of $z^*$ such that $\phi$ restricted on $U$ is continuous, then
\begin{equation}
\label{eqn:mean_gibbs}
    \lim_{\delta\to 0^+}\frac{\int z\exp\left(-\phi(z)/\delta\right)dz}{\int\exp{(-\phi(z)/\delta)dz}} = z^*\,.
\end{equation}
\end{theorem}
\begin{proof}
See Appendix~\ref{app:proof34}.
\end{proof}


When $z^*$ is nondegenerate, an error bound for the approximate proximal operator in \eqref{eqn:aprox} can be obtained according to the following theorem.
A similar result was derived in \cite[Chapter 9]{wong2001asymptotic}, but under stronger assumptions.
\begin{theorem}\label{thm_C2}
Assume that $\phi:\mathbb R^d\to \mathbb R$ is $p$-coercive for some $p>0$, and that $\phi$ has a unique nondegenerate global minimizer $z^*$.
Also, assume that there exists a neighborhood $U$ of $z^*$ such that $\phi$ restricted on $U$ is twice continuously differentiable.
Then
as $\delta\to 0^+,$
\[
\norm{\frac{\int z\exp\left(-\phi(z)/\delta\right)dz}{\int\exp{(-\phi(z)/\delta)dz}}-z^*} = \bigO(\delta).
\]
\end{theorem}
\begin{proof}
See Appendix~\ref{app:proof35}
\end{proof}

According to Theorems~\ref{thm:conv} and \ref{thm_C2}, evaluating the parameterized operator, $\prox_{tf}^\delta(x)$, as given in 
\eqref{eqn:aprox} for a small $\delta>0$ indeed provides a good approximation of the proximal operator. To illustrate this numerically, Figure~\ref{fig:prox_acc} displays the approximation errors for 
the two-dimensional Ackley function \cite{benchmark_Andrea}.
Figure~\ref{fig:2a} indicates that sufficiently accurate approximations may be obtained by choosing $\delta\le 0.5.$
Figure~\ref{fig:2b} indicates that for varied $\delta$ and $x,$ $\prox_{tf}^\delta(x)$ is closer to the global minimizer $x^*$ than $x,$ making it effective in an IPP method. Additionally, abundant numerical evidence demonstrating the effectiveness of approximating the proximal operator 
using MC estimates of $\prox^\delta_{tf}(x)$ can be found in \cite{osher2023hamilton, tibshirani2024laplace}.

\begin{figure}[H]
    \centering
    \subfloat[]{\includegraphics[width=0.48\textwidth]{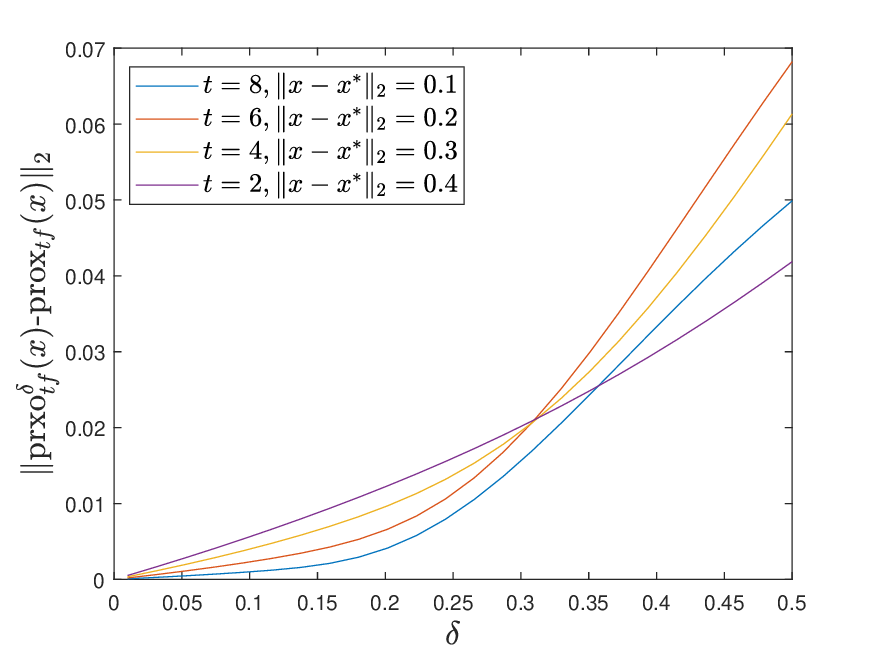} \label{fig:2a}}
    \hfill
    \subfloat[]{\includegraphics[width=0.48\textwidth]{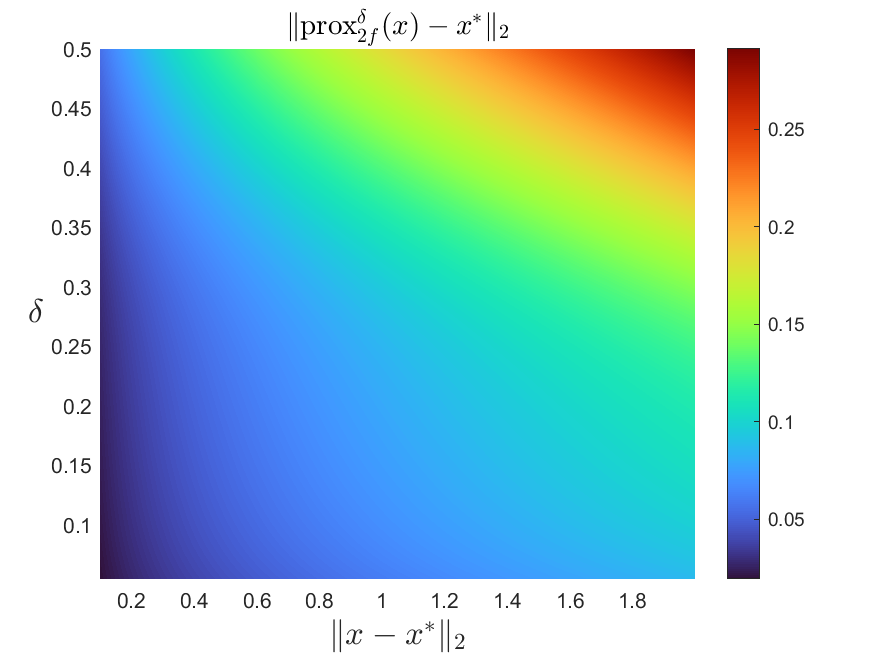} \label{fig:2b}}
    \vspace{-0.5cm}
    \caption{Approximations of \( \prox_{tf}(x) \) by \( \prox^\delta_{t f}(x) \) for Ackley function \cite{benchmark_Andrea}. (a)
    Approximation errors for varied $t$, $\delta$ and $x$; (b) Distances between \( \prox_{t f}^\delta(x) \) 
    and $x^*$ for varied $\delta$ and $x$, with $t = 2$.} 
    \label{fig:prox_acc}
\end{figure}


Of theoretical interest, when the proximal operator is multi-valued, the following corollary implies that 
 $\prox_{tf}^\delta(x)$ given in \eqref{eqn:aprox}
approximates a point in the convex hull of $\prox_{tf}(x)$.
\begin{corollary}\label{cor:convexHull}
Assume that $\phi:\mathbb R^d\to \mathbb R$ is $p$-coercive
for some $p>0$, and that $\phi$ has multiple nondegenerate global minimizers $z_1^*,\cdots,z_m^*$.
Also assume that, for each $j\in\Set{1,\cdots,m}$, there exists a neighborhood $U_j$ of $z_j^*$ such that $\phi$ restricted on $U_j$ is 
twice continuously differentiable.
Then, as $\delta\to 0^+,$
\[
\norm{\frac{\int z\exp\left(-\phi(z)/\delta\right)dz}{\int \exp{(-\phi(z)/\delta)dz}}-\bar z^*} = \bigO(\delta)\,,
\]
for some $\bar z^*\in \Set{\sum_{j=1}^m a_j z^*_j: a_j\ge 0 \textrm{ and } \sum_{j=1}^m a_j=1}.$
\end{corollary}
\begin{proof}
See Appendix~\ref{app:proof36}.
\end{proof}

\section{Tensor train for estimating the proximal operator}\label{sec:approx}
Based on results in Section~\ref{sec:analysis}, an inexact evaluation of the proximal operator can be obtained by estimating 
$\prox^\delta_{tf}(x)$ defined in \eqref{eqn:aprox}
for some small $\delta>0.$
Deterministic methods, such as the trapezoidal rule, require a function evaluation on each quadrature node, and the number of nodes grows exponentially as the dimension $d$ increases. Therefore, directly using these methods can be prohibitively expensive. Alternatively, randomized approaches, such as Monte Carlo (MC) integration,
may be used to compute the integrals in (\ref{eqn:aprox}). The MC-based method has been introduced in 
\cite{osher2023hamilton,heaton2024global,tibshirani2024laplace},
which may suffer from high variance and underflow errors in practice \cite{heaton2024global, beck2018underflow}.

In this section, we consider a new approach that is based on the usage of a \emph{tensor train} (TT) approximation algorithm 
\cite{oseledets2010tt, cross_error},
and analyze the associated error in estimating the proximal operator. 
By exploiting the Sobolev smoothness of the integrands, the TT-based method circumvents the curse of dimensionality, thereby improving the estimation accuracy.
Specifically, we first compute a low-rank TT approximation of the function
\begin{equation}\label{eqn:psi}
\psi := \exp\left(-\frac{f}{\delta}\right),
\end{equation}
and then use a quadrature rule to estimate (\ref{eqn:aprox}).
For computational purposes involving the TT approximation, we restrict the definition of $f$ 
on a bounded domain $\Omega\subset\Re^d$.

\subsection{Tensor train algorithms}\label{sec:tt}
In this subsection, we provide a brief review of TT algorithms and the associated error bound.
As the goal is to estimate the integrals in (\ref{eqn:aprox}) using TT approximation, we consider a $d$-dimensional mesh grid 
\[
\mathcal Z := \Set{z_1^{(i_1)}}_{i_1=1}^{n_1}\times \cdots\times\Set{z_d^{(i_d)}}_{i_d=1}^{n_d}\,,
\]
with each node  given by $\left(z_1^{(i_1)},\cdots,z_d^{(i_d)}\right)\in\Omega$. 
The discretization of a function $\psi$ on $\Z$, denoted by
$\psi_{\Zd}$, can be viewed as
a tensor of dimension $n_1\times\cdots \times n_d$, with entries given by the values of $\psi$ at node points.

A TT approximation of $\psi_{\Zd}$, denoted by $\psi_\TT$, is given by
\begin{align}\label{eqn:tt}
&\psi_\TT(i_1,\cdots,i_d)\\
:= &\sum_{\alpha_1=1}^{r_1} \cdots\sum_{\alpha_{d-1}=1}^{r_{d-1}} G_1(r_0,i_1, \alpha_1)G_2(\alpha_1, i_2, \alpha_2)\cdots G_d(\alpha_{d-1}, i_d\,, r_d)\,,\notag
\end{align}
where $r_0=r_d=1$ and $G_1,\cdots,G_d$ are called the \emph{cores}, with each $G_j$ of dimension $r_{j-1}\times n_j \times r_j$.
When the approximation is exact, the TT decomposition of a tensor is a generalization of the singular value decomposition (SVD) 
of a matrix and the cores are analogues to singular vectors.
The separable structure of $\psi_{\TT}$ reduces the computational cost of numerical integration from $\bigO(n^d)$ to $\bigO(dnr^2)$, where $n=\max\{n_1,\cdots,n_d\}$ and $r=\max\{r_1,\cdots,r_d\}$.

We include a standard result on the TT approximation error below.
\begin{theorem}
\label{Thm_TT_accuracy}
\cite[Thm. 2.2]{oseledets2010tt} For a tensor $\psi_{\Zd}$, define the \emph{unfolding matrices}
\begin{equation}\label{eqn:unfolding}
A_j:= \psi_{\Zd}(i_1,\cdots,i_j; i_{j+1},\cdots,i_d)\,,\quad j=1,\cdots,d-1\,,
\end{equation}
where the first $j$ indices enumerate the rows of $A_j$ and the last $d-j$ indices enumerate the columns.
There exists a tensor train approximation $\psi_\TT$ with ranks $\Set{r_j}_{j=1}^{d-1}$ such that
$\|\psi_{\Zd}-\psi_{\TT}\|_F \leq \sqrt{\sum_{j=1}^{d-1}\epsilon_j^2}$,
where 
$\epsilon_j = \min_{\textrm{rank} (B) \leq r_j}\|A_j-B\|_F$ for each $j$.
\end{theorem}
The theorem above implies that the error of the TT approximation can be made sufficiently small by choosing an appropriate rank. In particular, functions possessing certain Sobolev smoothness and underlying low-rank structures (e.g., Figure~\ref{fig:1d}) can be approximated with a relatively low TT rank \cite{Thm_TT_accuracy}; see also Table \ref{tab:tt_comparison}.

To efficiently construct a TT approximation $\psi_{\TT}$, we consider the randomized
TT cross algorithm \cite{oseledets2010tt, savostyanov2014quasioptimality}, which is based on the cross approximation of matrices.
Given a matrix $A$, a rank-r \emph{cross approximation} of $A$ is given by
$$A \approx A(:,J) A(I,J)^{-1} A(I,:)\,,$$
where $J$ is a subset of column indices of $A$ and $I$ is a subset of row indices, with $|I|=|J|=r$.
The index sets $I$ and $J$ are selected based on the \emph{maximum volume principle}, i.e., 
selecting the submatrix $A(I,J)$ that has the largest possible absolute value of the determinant.
The TT cross algorithm iteratively samples random row or column multi-indices and updates the tensor cores by performing cross approximation on sampled submatrices of unfolding matrices.
Details of the algorithm are summarized in Algorithm~\ref{alg:TTcross}, including two subroutines \texttt{TT-Cross-Right-To-Left-Sweep}
and \texttt{TT-Cross-Left-To-Right-Sweep}.
The subroutine \texttt{TT-Cross-Right-To-Left-Sweep} performs cross approximation on submatrices of the unfolding matrices $A_j$ given 
in \eqref{eqn:unfolding} from
$j=d-1$ to $j=1$, and \texttt{TT-Cross-Right-To-Left-Sweep} performs cross approximation on submatrices of $A_j$ from
$j=1$ to $j=d-1$.
The per-iteration cost of the TT cross algorithm is roughly $\bigO(dr^3)$ flops and $\bigO(dr^2)$ function evaluations. 
The algorithm obtains a TT approximation by evaluating only a small number of entries from the original tensor, 
without ever storing the full tensor, thus substantially reducing computational and storage costs.

\begin{algorithm}[H]
\caption{TT Cross Algorithm \cite{oseledets2010tt}}
\label{alg:TTcross}
\begin{algorithmic}[1]
    \STATE \textbf{Input:} Black-box tensor $\psi_{\Zd}\in \mathbb{R}^{n_1\times \cdots \times n_d}$, tolerance $\tau_{stop}>0$
    \STATE \textbf{Output:} $\psi_{\TT}$ with cores $G_1,\cdots,G_d$
    \STATE Randomly choose sets of column multi-indices $J_1,\cdots,J_{d-1}$
    \STATE $G_1,\cdots,G_d,I_1,\cdots,I_{d-1} \gets \text{TT-Cross-Left-To-Right-Sweep}(\psi_{\Zd}, J_1,\cdots,J_{d-1})$
    \STATE $H_1,\cdots,H_d \gets 0$
    \WHILE{$\|H_1\cdots H_d - G_1\cdots G_d\|_F < \tau_{stop} \|G_1\cdots G_d\|$}
        \STATE $\widehat{I}_1,\cdots,\widehat{I}_{d-1} \gets I_1,\cdots,I_{d-1}$ extended with random row multi-indices
        \STATE $H_1,\cdots,H_d,J_1,\cdots,J_{d-1} \gets \text{TT-Cross-Right-To-Left-Sweep}(\psi_{\Zd}, \widehat{I}_1,\cdots,\widehat{I}_{d-1})$
        \STATE $\widehat{J}_1,\cdots,\widehat{J}_{d-1} \gets J_1,\cdots,J_{d-1}$ extended with random column multi-indices
        \STATE $G_1,\cdots,G_d,I_1,\cdots,I_{d-1} \gets \text{TT-Cross-Left-To-Right-Sweep}(\psi_{\Zd}, \widehat{J}_1,\cdots,\widehat{J}_{d-1})$
    \ENDWHILE
\end{algorithmic}
\end{algorithm}


If two tensors are both represented in the TT-format \eqref{eqn:tt}, the Hadamard product can be performed directly
using such representation, which is useful in designing an efficient TT-based IPP algorithm in Section~\ref{sec:ttIPP} below.
When reducing the parameter $\delta$ to $\delta/2$ in \eqref{eqn:aprox},  we simply need to approximate $\psi^2$, whose TT approximation can be obtained 
as the Hadamrd product of $\psi_{\TT}$ and itself.
As derived in \cite{oseledets2011tensor}, the Hadamard product of the two tensors can be written explicitly as 
\[
\psi_{\TT}\circ \psi{\TT} = (G_1\otimes G_1)\cdot(G_2\otimes G_2)\cdots (G_d\otimes G_d)\,,
\]
where $\otimes$ denotes the Kronecker product.
Note that the resulting tensor can be written out explicitly in this manner and does not require any extra function evaluations. 
Moreover, as performing the Hadamard product leads to an increase in the ranks, a rounding procedure, \emph{TT-rounding}, introduced in 
\cite{oseledets2011tensor} can be employed to reduce the ranks while preserving the accuracy of the TT approximation within a specified tolerance. 
The TT-rounding procedure involves orthogonalization and performing truncated SVDs on unfolded tensor cores, with a computational complexity of 
$\bigO(dnr^3).$

The estimation of the proximal operator in \eqref{eqn:aprox} involves the integrand
$\tilde \psi(z):= \psi(z)\eta(z)$, where $\eta(z):= \exp{\left(-\norm{z-x}^2/(2t\delta)\right)}$ for a fixed $x$.
The discretization of the function $\eta(z)$ can be easily represented in a rank-$1$ TT format due to its separable structure, with cores given by
\[
u_j := \left[\exp\left(-\frac{1}{2t\delta}\|z_{j,1}-x_j\|^2 \right),\cdots,\exp\left(-\frac{1}{2t\delta}\|z_{j,n_j}-x_j\|^2 \right) \right]^T
\]
for $j=1,\cdots,d$, where $\{z_{j,k}\}_{k=1}^{n_j}$ are nodal points of $\Z$ along the $j$-th dimension.  
The integrals in \eqref{eqn:aprox} can then be estimated using a quadrature rule.
Let $\Set{w_{j,k}}_{k=1}^{n_j}$ be the quadrature weights associated with nodal points along the $j$-th dimension.
The denominator in \eqref{eqn:aprox} is approximated by a discrete sum
\begin{align}\label{eqn:tt_integral}
&\int_\Omega\exp\left(-\frac{\phi(z)}{\delta}\right)dz = \int_{\Omega}\exp\left(-\frac{f(z)}{\delta}\right)\exp\left(-\frac{\|z-x\|_2^2}{2t\delta}\right)dz \notag\\
\approx& \prod_{j=1}^{d}\left(\sum_{i_j=1}^{n_j}L_j(i_j)\right),\text{ where } L_j(i_j) = G_j(i_j)u_j(i_j)w_{j,i_j}\,,
\end{align}
with $G_j(i_j)=G_j(:,i_j,:)$ is of dimension $r_{j-1}\times r_j$ for each $j$ where $G_j$ is the core for the TT approximation for $\exp(-f/\delta)$ as defined in \eqref{eqn:tt}. 
Similarly, the $j$-th component of the numerator in \eqref{eqn:aprox} is approximated by
\begin{align}\label{eqn:tt_zintegral}
&\left[\int_\Omega z\exp\left(-\frac{\phi(z)}{\delta}\right)dz\right]_j \\
\approx& \left(\sum_{i_1=1}^{n_1}L_1(i_1)\right) \cdots  \left( \sum_{i_j=1}^{n_j} z_{j,i_j}L_j(i_j) \right)\cdots \left(\sum_{i_d=1}^{n_d}L_d(i_d)\right)\notag\,.
\end{align}

\subsection{Error analysis}\label{sec:error}
In this subsection, we derive error bounds for estimating the proximal operator based on (\ref{eqn:aprox})
using the TT approximation.
For simplicity, we consider the case where $\Omega=[0,1]^d$, and assume that $n_1=\cdots =n_d=n$ with the mesh size $h=1/n$.
The results can be easily extended to a general bounded domain $\Omega\subset\Re^d$.
The error is comprised of two parts: the error in the TT approximation and the error in the numerical integration.

As discussed in Section~\ref{sec:tt}, the error in the TT approximation can be made arbitrarily small by allowing sufficiently large ranks. In particular, for functions possessing certain Sobolev smoothness and underlying low-rank structures (as illustrated in Figure~\ref{fig:1d}), a relatively low TT rank is sufficient. To analyze the error in each component of the numerical computation, we assume that  
\begin{equation}\label{eqn:epsilonTT}
 \|\psi_{\TT}-\psi_{\mathcal{\Zd}}\|_F\leq {\epsilon_{\TT}}\,,
\end{equation}
where $\psi_\Zd$ is the discretization of $\psi$ on the mesh grid $\mathcal{Z}$.
The value of $\epsilon_{\TT}$ depends on the choice of the termination tolerance $\tau_{stop}$ in Algorithm~\ref{alg:TTcross}.
Detailed error analysis
on TT cross approximation can be found in \cite{cross_error, oseledets2010tt,Thm_TT_accuracy}. 

Now we consider the error in applying a quadrature rule on the TT approximations of the integrands in (\ref{eqn:aprox}).
Let $\left(\int \exp\left(-\phi(z)/\delta\right) \, dz\right)_{\TT}$ and\\ $\left(\int z \exp\left(-\phi(z)/\delta\right) \, dz\right)_{\TT}$ denote the TT approximations of the integrals \\$\int \exp\left(-\phi(z)/\delta\right) \, dz$ (see \eqref{eqn:tt_integral}) and $\int z \exp\left(-\phi(z)/\delta\right) \, dz$ (see \eqref{eqn:tt_zintegral}), respectively. 
Then the TT approximation of the proximal operator is given by
\begin{equation}\label{eq:x_TT}
\prox_{tf}(x)\approx \frac{\left(\int_\Omega z\exp\left(-\phi(z)/\delta\right)dz\right)_{\TT}}
{\left(\int_\Omega\exp\left(-\phi(z)/\delta\right)dz\right)_{\TT}} :=(\prox^\delta_{tf}(x))_{\TT}\, .
\end{equation}

First, we consider the case where $f\in C^2(\Omega)$. When the trapezoidal rule is used, we have the following standard result:
\begin{align}\label{eq:int_error1}
&\norm{\int_\Omega z\exp\left(-\frac{\phi(z)}{\delta}\right)dz-\left(\int_\Omega z\exp\left(-\frac{\phi(z)}{\delta}\right)dz\right)_{\Zd}}\\
\le &\frac{d h^2}{12}\max_{z,k,l}
\norm{\frac{\partial^2 \left(z\exp\left(-\phi(z)/\delta\right)\right)}{\partial z_k \partial z_{l}}}\,,\notag
\end{align}
where $(\cdot)_{\Zd}$ denotes the numerical quadrature over the mesh $\Z$ for estimating the integral.
Similarly,
\begin{align}\label{eq:int_error2}
&\norm{\int_\Omega \exp\left(-\frac{\phi(z)}{\delta}\right)dz-\left(\int_\Omega \exp\left(-\frac{\phi(z)}{\delta}\right)dz\right)_{\Zd}}\\
\le &\frac{d h^2}{12}\max_{z,k,l}
\norm{\frac{\partial^2 \left(\exp\left(-\phi(z)/\delta\right)\right)}{\partial z_k \partial z_{l}}}\,.\notag
\end{align}
It follows that, for $\delta\in (0,1)$, the errors in \eqref{eq:int_error1} and \eqref{eq:int_error2} are bounded by
$\frac{c_1 dh^2}{12\delta^2}\exp\left(-\frac{\phi_{\min}}{\delta}\right)$,
where $c_1$ is a constant that depends on the magnitude of the first and second-order partial derivatives of $\phi$ on $\Omega$, and 
$$\phi_{\min}=\min_{z\in\Omega}\phi(z)\ge f(x^*)\,.$$

Thus, by the Cauchy-Schwarz inequality,
\begin{align*}
&\norm{\int_\Omega z\exp\left(-\frac{\phi(z)}{\delta}\right)dz-\left(\int_\Omega z\exp\left(-\frac{\phi(z)}{\delta}\right)dz\right)_{\TT}}\\
\le & \norm{\int_\Omega z\exp\left(-\frac{\phi(z)}{\delta}\right)dz-\left(\int_\Omega z\exp\left(-\frac{\phi(z)}{\delta}\right)dz\right)_{\Z}}\\
& +\norm{\left(\int_\Omega z\exp\left(-\frac{\phi(z)}{\delta}\right)dz\right)_{\Zd}-\left(\int_\Omega z\exp\left(-\frac{\phi(z)}{\delta}\right)dz\right)_{\TT}}\\
\le & \frac{c_1 dh^2}{12\delta^2}\exp\left(-\frac{\phi_{\min}}{\delta}\right)+ \norm{\psi_{\Zd}-\psi_{\TT}}_F\left(\sum_{j=1}^{n^d}\left(\norm{z^{(j)}}w^{(j)}\right)^2\right)^{1/2}\\
\le & \frac{c_1 dh^2}{12\delta^2}\exp\left(-\frac{\phi_{\min}}{\delta}\right)+ c_2\epsilon_{\TT} h^{d/2}\,,
\end{align*}
where $\Set{z^{(j)}}$ and $\Set{w^{(j)}}$ denote the quadrature nodes and weights,
$c_2$ is the approximation to integral $\int_\Omega \norm{x}^2dx\le 1$ which is a constant.
Similarly,
\begin{align*}
&\norm{\int_\Omega \exp\left(-\frac{\phi(z)}{\delta}\right)dz-\left(\int_\Omega \exp\left(-\frac{\phi(z)}{\delta}\right)dz\right)_{\TT}}\\
\le &\frac{c_1 dh^2}{12\delta^2}\exp\left(-\frac{\phi_{\min}}{\delta}\right)+ \tilde c_2\epsilon_{\TT} h^{d/2}\,,
\end{align*}
where $\tilde c_2\approx |\Omega|=1.$


The TT estimation error of the proximal operator satisfies
\begin{align}
 \label{eqn:discrete_1}
   &  \norm{\prox_{tf}(x)-\left(\prox_{tf}^\delta(x) \right)_{\TT}}\\
\leq &\norm{\prox_{tf}(x)-\prox_{tf}^\delta(x)}+\norm{\prox_{tf}^\delta(x)-\left(\prox_{tf}^\delta(x) \right)_{\TT}}\,.\notag
\end{align}
By Theorem~\ref{thm_C2}, for $0<\delta\ll 1$, the first term above is $\bigO(\delta)$, i.e.,
\[
\norm{\prox_{tf}(x)-\prox_{tf}^\delta(x)}\le C_1\delta
\]
for some constant $C_1$ that depends on the homeomorphism $T$ in Lemma~\ref{lem:morse}.
To estimate the second error term, let
\[
\beta_1 = \int_\Omega z\exp\left(-\frac{\phi(z)}{\delta}\right)dz\,, \quad \beta_2 =  \int_\Omega \exp\left(-\frac{\phi(z)}{\delta}\right)dz\,,
\]
and let $\tilde\beta_1$, $\tilde\beta_2$ as their TT approximations respectively:
\[
\tilde\beta_1 = \left(\int_\Omega z\exp\left(-\frac{\phi(z)}{\delta}\right)dz\right)_{\TT},\quad 
\tilde\beta_2 = \left(\int_\Omega \exp\left(-\frac{\phi(z)}{\delta}\right)dz\right)_{\TT}.
\]
The magnitude of the second term in \eqref{eqn:discrete_1}  is then bounded by  
\begin{align*}
&\norm{\frac{\beta_1}{\beta_2} - \frac{\tilde{\beta}_1}{\tilde{\beta}_2}} \leq 
\norm{\frac{\tilde{\beta}_1}{\tilde{\beta}_2} - \frac{\tilde{\beta}_1}{\beta_2}} + \norm{ \frac{\tilde{\beta}_1}{\beta_2} - \frac{\beta_1}{\beta_2}}
\leq  \norm{\frac{\tilde{\beta}_1}{\tilde{\beta}_2}} \left|\frac{\beta_2-\tilde{\beta}_2}{\beta_2}\right| + \norm{\frac{\beta_1-\tilde{\beta}_1}{\beta_2}} \\
&\le\left(\norm{\frac{\tilde{\beta}_1}{\tilde{\beta}_2}}+1\right)
\left(\frac{c_1 dh^2}{12\delta^2\int_\Omega \exp\left(-\tilde \phi(z)/\delta\right)dz}
+\frac{\max\{\tilde c_2,c_2\}\epsilon_{\TT} h^{d/2}\exp\left(\frac{\phi_{\min}}{\delta}\right)}{\int_\Omega \exp\left(-\tilde \phi(z)/\delta\right)dz}\right),
\end{align*}
where $\tilde \phi(z) = \phi(z)-\phi_{\min}.$
As derived in the proof of Theorem~\ref{thm_C2},
\[
\int_{\Omega} \exp{\left(-\frac{\tilde \phi(z)}{\delta}\right)}dz = \frac{(2\pi\delta)^{d/2}}{\sqrt{\det(\Hess\tilde \phi(z^*))}}+\bigO(\delta^{d/2+1})
\]
for small $\delta>0.$ Also notice
\(
 \norm{\tilde{\beta}_1/\tilde{\beta}_2} \le \max_{j} \norm{x^{(j)}}\le 1\,.
\)
It follows that, for $0<\delta\ll 1,$
\[
\norm{\frac{\beta_1}{\beta_2} - \frac{\tilde{\beta}_1}{\tilde{\beta}_2}} \le \frac{C_2 dh^2}{\delta^{2+d/2}}+ 
\frac{C_3\epsilon_{\TT} h^{d/2}}{\delta^{d/2}}\exp\left(\frac{\phi_{\min}}{\delta}\right)\,,
\]
where $C_2$ and $C_3$ are constants that depend on the magnitude of the first and second-order derivatives of $f$ on $\Omega$.

The results above are summarized in the following proposition.
\begin{proposition}
\label{prop:C2}
Assume that $f\in C^2(\Omega)$ and that $z^*:=\prox_{tf}(x)$ is the unique nondegenerate global minimizer of $\phi$.
For $0<\delta\ll 1$, the error in estimating $\prox_{tf}(x)$ using TT approximation and Trapezoidal rule is given by
\begin{equation}\label{eq:int_tt_error}
\norm{\prox_{tf}(x)-\left(\prox_{tf}^\delta(x) \right)_{\TT}}\le C_1\delta+
\frac{C_2 dh^2}{\delta^{2+d/2}}+ \frac{C_3\epsilon_{\TT} h^{d/2}}{\delta^{d/2}}\exp\left(\frac{\phi_{\min}}{\delta}\right),
\end{equation}
where $C_1, C_2, C_3$ are constants that are independent of the choices of $\Set{\delta,\epsilon_{\TT},h}$.
In particular, as $\delta\to 0^+$, if
\begin{equation}\label{eq:hepsilon}
h = \bigO(\delta^{\max\{d,2\}/4+3/2}) \quad \textrm{ and } \quad \epsilon_{\TT}\le \exp\left(-\frac{\phi_{\min}}{\delta}\right)
\end{equation}
then
$\norm{\prox_{tf}(x)-\left(\prox_{tf}^\delta(x) \right)_{\TT}} = \bigO(\delta).$
\end{proposition}

For high-dimensional problems, the bound on the mesh size $h$ given in (\ref{eq:hepsilon}) is unrealistic. 
This restriction is due to the theoretical error bounds on the numerical integration given in \eqref{eq:int_error1}-\eqref{eq:int_error2}.
Alternatively, we consider the case where $f$ lies in a \emph{Sobolev space} given by
\[
H^s(\Omega):=\Set{g\in L^2(\Omega):\norm{g}_s^2:=\left(\sum_{\tau=0}^s\norm{g^{(\tau)}}_{L^2(\Omega)}^2\right)<\infty}\,.
\]
When $s \ge 2$, \cite[Theorem 4.5 ]{Sobolev_integral} implies 
\begin{equation}
\label{eqn:error_Hs}
    \left\|\int_\Omega z\exp\left(-\frac{\phi(z)}{\delta}\right)dz-
    \left(\int_\Omega z\exp\left(-\frac{\phi(z)}{\delta}\right)dz\right)_{\Z}\right\|\le C\frac{d(\log n)^{\frac{s}{2}+\frac{1}{4}}}{n^s}
\end{equation}
where $C$ is a constant dependent on the Sobolev norm of the integrand.
Following similar arguments as for the previous case, the following result can be derived.
\begin{corollary}\label{cor:TT_error}
  Assume that $f\in H^{s}(\Omega)$ for $s\ge 2$ and that $z^*:=\prox_{tf}(x)$ is the unique nondegenerate global minimizer of $\phi$. If
\begin{equation}\label{eqn:hepsilon2}
      h = \mathcal{O}\left(\delta^{\frac{d+2}{2s}+1}\right)~\textrm{ and }~ \epsilon_{\TT}\le \exp\left(-\frac{\phi_{\min}}{\delta}\right).
\end{equation}
Then the error in in estimating $\prox_{tf}(x)$ using TT approximation and Trapezoidal rule satisfies
\begin{equation*}
\norm{\prox_{tf}(x)-\left(\prox_{tf}^\delta(x) \right)_{\TT}} = \tilde\bigO(\delta)\,,
\end{equation*}
where $\tilde\bigO(\cdot)$ omits polylogarithmic factors.
\end{corollary}

By \eqref{eq:hepsilon} and \eqref{eqn:hepsilon2}, if $\phi_{\min}\le 0$, it is sufficient to require the TT approximation error $\epsilon_{\TT}\le 1$ to guarantee the desired
accuracy in estimating the proximal operator. In practice, if we have an estimate of $\phi_{\min}$, i.e., $c\approx \phi_{\min}$ for some constant $c$, we 
may shift the original $\phi$ and estimate $\prox_{tf}(x)$ using the equivalent formula
\[
\prox_{tf}(x) = \frac{\int_{\Omega} z\exp\left(-\tilde \phi(z)/\delta\right)dz}{\int_{\Omega}\exp{(-\tilde \phi(z)/\delta)dz}}\,,
\]
where $\tilde \phi(z) = \phi(z)-c,$ with $\tilde \phi_{\min} = \phi_{\min}-c\approx 0.$ 
This shift is not mandatory but may enhance numerical stability.


\section{Two practical algorithms}\label{sec:IPP_alg}
\subsection{The TT-IPP algorithm}\label{sec:ttIPP}
In this section, we propose a practical IPP algorithm, TT-IPP, which utilizes the TT approximation 
to estimate the proximal operator 
\[
\prox_{tf}(x)\approx \left(\prox_{tf}^\delta(x)\right)_{\TT},
\]
and allows the parameter $\delta$ to decrease adaptively.
The TT estimate $\left(\prox_{tf}^\delta(x)\right)_{\TT}$ is computed as described in Section~\ref{sec:approx}.
Details of TT-IPP are summarized in Algorithm~\ref{alg:ttIPP}.
An initial TT approximation of the function $\psi:=\exp{\left(-f/\delta_0\right)}$ is computed first, and the parameter $\delta_k$
is decreased adaptively based on whether the current iterate achieves a sufficient function decrease as written in 
Line~\ref{line:if_decrease} of Algorithm~\ref{alg:ttIPP}. 
For iterations where $\delta_k$ is not reduced, the previous TT approximation is reused.
When $\delta_k$ is reduced by half, we compute the Hadamard product to efficiently update the TT approximation without 
the need of any additional function evaluation. 
According to the error analysis in Section~\ref{sec:error}, TT estimates of the proximal operators are accurate if the mesh size is small relative to the value of $\delta.$
When the mesh size $h_k$ is too large relative to the current parameter $\delta_k$, the
mesh size is reduced and the TT approximation is refined over the new mesh grid. When using a uniform mesh, previous function evaluations may be reused in refining the TT approximation in Line~\ref{line:refineTT} if function evaluations are expensive.
Additionally, according to Theorem \ref{thm:conv}, a reasonable initial guess of the global minimizer can be obtained by approximating the integrals in \eqref{eqn:mean_gibbs} if substituting $\phi$ with the original objective function $f$, i.e.
\begin{equation}\label{eq:initial}
x^0\approx \frac{\int_\Omega x\exp\left(-f(x)/\delta_0\right)dx}{\int_\Omega \exp{(-f(x)/\delta_0)dx}}\,.
\end{equation}
The TT-estimate of \eqref{eq:initial}, based on the initial TT approximation of $\exp{\left(-f/\delta_0\right)}$,
can serve as a warm start of TT-IPP at a cost equivalent to a single iteration of the algorithm.

\begin{algorithm}
\caption{Tensor-Train Inexact Proximal Point Method (TT-IPP)}\label{alg:ttIPP}
\begin{algorithmic}[1] 
    \STATE \textbf{Input:} $x^0\in\Re^d$, $0<\delta_0< 1$, $0<\eta_-<1<\eta_+$, $0<\theta_1\le\theta_2<1$, $\bar\epsilon>0$, 
    $0<\eta<1$, $T>0$, $0<\tau\le t_0\le T$, $h_0>0$, 
    $\gamma>1$, $C>0$, $m\ge 1$, $k_{\max}>0$, and $\epsilon_{stop}>0$
    \STATE $k \gets 0$
    \STATE Compute $\left(\exp\left(\frac{-f}{\delta_{0}}\right)\right)_{\TT}$, a TT approximation of $\exp\left(\frac{-f}{\delta_0}\right)$ 
    on a mesh grid of size $h_0$ by Algorithm~\ref{alg:TTcross}
    \WHILE{$k < k_{\max}$}
        \STATE $x^{k+1} \gets \left(\prox_{t_k f}^{\delta_k}(x^k)\right)_{\TT}$
        \IF{$k \ge m-1$ \textbf{and} $f(x^{k+1}) > \max\{f(x^k), f(x^{k-1}),\dots,f(x^{k-m+1})\} - \eta/k$} \label{line:if_decrease}
            \STATE $\delta_{k+1} \gets \delta_k / 2$
            \STATE $\left(\exp\left(\frac{-f}{\delta_{k+1}}\right)\right)_{\TT} \gets \left(\exp\left(\frac{-f}{\delta_{k}}\right)\right)_{\TT}$ 
            using Hadamard product
            \IF{$h_k > C \delta_k^\gamma$}
                \STATE $h_{k+1} \gets h_k / 2^{\lfloor \gamma \rfloor}$
                \STATE Update $\left(\exp\left(\frac{-f}{\delta_{k+1}}\right)\right)_{\TT}$ on the refined mesh by Algorithm~\ref{alg:TTcross} \label{line:refineTT}
            \ENDIF
        \ELSE
            \STATE $\delta_{k+1} \gets \delta_k$, $h_{k+1} \gets h_k$
            \STATE $\left(\exp\left(\frac{-f}{\delta_{k+1}}\right)\right)_{\TT} \gets \left(\exp\left(\frac{-f}{\delta_{k}}\right)\right)_{\TT}$
        \ENDIF
        \STATE Determine $t_{k+1}$ using Lines~\ref{line:qk}--\ref{line:tk} of Algorithm~\ref{alg:IPP}
        \STATE $k \gets k + 1$
        \IF{$\norm{x^{k+1} - x^k} < \epsilon_{stop}$}
            \STATE \textbf{Break}
        \ENDIF
    \ENDWHILE
    \STATE \textbf{Output:} last iterate $x^k$
\end{algorithmic}
\end{algorithm}

The convergence of TT-IPP is a corollary of the convergence of IPP methods proved in Theorem~\ref{thm:IPP}.
\begin{corollary}\label{cor:ttIPP}
For $f:\Omega\subset\Re^d\to\Re$, where $\Omega$ is a bounded domain,
suppose Assumptions~\ref{assum1} -- \ref{assum3} hold, and one of the conditions listed in Proposition~\ref{prop:single_prox} holds.
Let $\Set{x^k}_{k\ge 0}$ be the sequence of iterates generated by Algorithm~\ref{alg:ttIPP}. 
If $T>0$ and $\gamma>1$ chosen for Algorithm~\ref{alg:ttIPP}
are sufficiently large, and the TT approximation error $\epsilon_{\TT}$ in \eqref{eqn:epsilonTT} is sufficiently small,
then $\Set{x^k}_{k\ge 0}$ converges to $x^*$ as $k\to\infty.$
In particular, for $f\in H^s{(\Omega)}\cap C^2(\Omega)$, it is sufficient if $\gamma > (d+2)/(2s)+1$, and 
$\epsilon_{\TT}$ satisfies \eqref{eqn:hepsilon2} at each iteration.
\end{corollary}
\begin{proof}
Assume that  $\lim\limits_{k\to\infty}\delta_k > 0$, then the condition in Line~\ref{line:if_decrease} holds for only finitely many iterations.
It follows that there exists a constant $K>0$ such that, for all $k>K$,
\[
f(x^{k+1}) \le \max\{f(x^k), f(x^{k-1}),\cdots,f(x^{k-m+1})\}-\eta/k\,.
\]
It can be shown by induction that, for all $k\ge m-1,$ 
\begin{align}
\label{eq:fdecrease}
f(x^{k+1})\le\max \bigg\{&f(x^0)-\sum_{j\in\K_0\cap\K}\frac{\eta}{j}\,,\\
&f(x^1)-\sum_{j\in \K_1\cap\K} \frac{\eta}{j}\,,
\cdots,\,f(x^{m-1})-\sum_{j\in \K_{m-1}\cap\K} \frac{\eta}{j}\bigg\}-c_k\,,\notag
\end{align}
where $\K=\Set{1,\cdots,k}$ and 
\[
\K_0=\Set{lm-1}_{l=1}^{\infty}, \,\,\K_1=\Set{lm}_{l=1}^{\infty},\,\,\cdots,\,\,
 \K_{m-1}=\Set{(l+1)m-2}_{l=1}^{\infty}\,.
\]
The right-hand side of \eqref{eq:fdecrease} decreases to $-\infty$ as $k\to\infty$, which
contradicts the assumption that $f$ is bounded below by $f(x^*).$  Hence, $\lim_{k\to\infty}\delta_k = 0$.
As shown in the proof of Theorem~\ref{thm:IPP}, $t_k = T$ for all $k$ sufficiently large.
If any of the conditions listed in Proposition~\ref{prop:single_prox} holds for $f$ and $T>0$ is sufficiently large,
then the proximal operator $\prox_{t_k f}(x^k)$ is single-valued for all $k$ sufficiently large.
According to Theorem~\ref{thm:conv} and the error analysis in Section~\ref{sec:error}, if $\gamma>1$ in Algorithm~\ref{alg:ttIPP} is sufficiently large and
the TT approximation error $\epsilon_{\TT}$ is sufficiently small, then
\eqref{eq:prox_error} holds. In particular, for $f\in H^s{(\Omega)}\cap C^2(\Omega)$,
if $\gamma > (d+2)/(2s)+1$,  and 
$\epsilon_{\TT}$ satisfies \eqref{eqn:hepsilon2} at each iteration, then by Corollary~\ref{cor:TT_error}, \eqref{eq:prox_error} holds.
It follows from Theorem~\ref{thm:IPP} that, by choosing $T>0$ sufficiently large, 
$\Set{x^k}_{k\ge 0}$ is guaranteed to converge to $x^*$ as $k\to\infty.$
\end{proof}

While a fine mesh ensures the theoretical convergence of TT-IPP according to the above corollary, using a relatively coarse mesh improves computational efficiency. The error bounds in Section~\ref{sec:error} are derived for worst-case scenarios \cite{quarteroni2010numerical}
and are often larger than actual errors observed in practice. Therefore, in numerical experiments presented in Section~\ref{sec:experiments}, 
we heuristically select the control parameters for TT-IPP, guided by the theoretical error bounds.

\subsection{The MC-IPP algorithm}
In this section, we propose a practical IPP algorithm, MC-IPP, which is based on the Monte Carlo (MC) estimates of proximal operators.
The MC estimate of a proximal operator is given by
\begin{equation}
    \left(\prox_{t f}^{\delta}(x)\right)_{\MC} := 
\frac{\sum_{i=1}^N z^i\exp({-f(z^i)/\delta})}{\sum_{i=1}^N \exp({-f(z^i)/\delta})}\approx \prox_{tf}(x)
\end{equation}
for a sample size $N\ge 1$ and sample points $z^i\overset{\text{iid}}\sim \mathcal{N}(x, \delta t I_d)$
from the Gaussian distribution centered at \( x \).

To motivate the design of this algorithm, we provide an analysis of the MC sample complexity under the assumptions in 
Theorem~\ref{thm_C2} on the function $\phi$ given in \eqref{eqn:phi}.
Without loss of generality, we assume $\phi_{\min}=0.$
Define
\[
\beta_1:= \int z\exp\left({-\frac{\phi(z)}{\delta}}\right) dz \,\textrm{ and }\,
\hat\beta_1 :=\frac{(2\pi t\delta)^{d/2}}{N}\sum_{i=1}^N z^i\exp\left({-\frac{f(z^i)}{\delta}}\right)\approx \beta_1\,,
\]
Then by standard results of MC integration \cite{lemieux2009monte},
the expectation of $\hat \beta_1$ is $\expect{\hat \beta_1} = \beta_1$ and $\norm{\beta_1}=\bigO\left(\delta^{d/2}\right)$. Moreover,
following similar arguments as in the proof of Theorem~\ref{thm_C2}, the covariance satisfies
\begin{align*}
\norm{\Cov{\hat \beta_1}}_2 \le &\frac{1}{N} \int \norm{(2\pi t\delta)^{d/2}z\exp\left(-\frac{f(z)}{\delta}\right)-\beta_1 e}^2
\frac{\exp{\left(-\frac{\norm{z-x}^2}{2t\delta}\right)}}{(2\pi t\delta)^{d/2}}dz \\
= & \frac{1}{N} \int (2\pi t\delta)^{d/2} \norm{z}^2 \exp{\left(-\frac{2f(z)}{\delta}-\frac{\norm{z-x}^2}{2t\delta}\right)}dz-\frac{\|\beta_1\|^2}{N}\\
=& \bigO\left(\frac{\delta^{d}}{N}\right).
\end{align*}
Similarly, define
\[
\beta_2:= \frac{1}{N}\int \exp\left(-\frac{\phi(z)}{\delta}\right) dz \quad\textrm{ and }\quad \hat\beta_2 :=\frac{(2\pi t\delta)^{d/2}}{N}\sum_{i=1}^N \exp\left({-\frac{f(z^i)}{\delta}}\right)\,,
\]
then $\expect{\hat \beta_2} = \beta_2= \bigO\left(\delta^{d/2}\right)$ and $\Var{\hat \beta_2} = \bigO\left({\delta^{d}}/{N}\right)$.
Notice that, for arbitrary $\gamma\in (0,1)$, if
$\norm{\beta_1-\hat\beta_1} =\bigO(\delta^{d/2+\gamma})$ and
$\left|\beta_2-\hat\beta_2\right|=\bigO(\delta^{d/2+\gamma})$,
then
\[
\norm{\prox_{tf}(x)-\left(\prox_{t f}^{\delta}(x)\right)_{\MC}}\le \bigO(\delta)+
\norm{\frac{\beta_1}{\beta_2}-\frac{\hat \beta_1}{\hat\beta_2}}=\bigO(\delta^\gamma)\,.
\]
By Chebyshev's inequality, 
\[
\pr{\norm{\beta_1-\hat\beta_1}>\delta^{d/2+\gamma}}\le \norm{\Cov{\hat \beta_1}}_2/(\delta^{d+2\gamma})
\]
and
\[
\pr{\left|\hat\beta_2-\beta_2\right|>\delta^{d/2+\gamma}}\le \Var{\hat \beta_2}/(\delta^{d+2\gamma})\,,
\]
which implies \[
\pr{\norm{\prox_{tf}(x)-\left(\prox_{t f}^{\delta}(x)\right)_{\MC}}>\delta^\gamma} \le \bigO\left(\frac{1}{N\delta^{2\gamma}}\right).
\]
Therefore, for $\gamma\in (0,1/4)$, with a sample size 
$N = \bigO(\delta^{-1/2})$,
\[
\pr{\norm{\prox_{tf}(x)-\left(\prox_{t f}^{\delta}(x)\right)_{\MC}}>\delta^\gamma} \le \bigO\left(\delta^{1/2-2\gamma}\right)\,.
\]
The above computation suggests that a fixed sample size suffices to approximate the proximal point with a prescribed accuracy. However, the accuracy of approximation still depends on $d$ at least linearly due to the magnitude of the $ d$-dimensional Euclidean norm, and hence the variance of the estimate may still be large in the high-dimensional case. To address this issue and reduce sample complexity, variance reduction techniques can be employed \cite{lepage1980vegas,reddi2016stochastic,lemieux2009monte}.
Here we consider a simple variance reduction technique called the exponentially weighted moving average (EWMA) \cite{ross2009probability, kingma2014adam},
which computes
\[
x^{k+1} = \alpha \left(\prox_{t_k f}^{\delta_k}(x^k)\right)_{\MC}+(1-\alpha)x^{k}
\]
for a damping parameter $\alpha\in(0,1).$ EWMA reduces the variances of $\hat\beta_1$ and $\hat\beta_2$ to approximately
$\bigO\left(\frac{\delta^{d}\alpha}{N (2-\alpha) }\right)$, thereby reduces the required sample size to
\begin{equation}\label{eq:sampleComp}
N = \bigO\left(\frac{\delta^{-1/2}\alpha}{2-\alpha}\right).
\end{equation}
In practice, we observe in our numerical experiments in Section~\ref{sec:experiments} that a sample size smaller than the scale of \eqref{eq:sampleComp} is needed
to achieve a desirable empirical accuracy.

Details of the MC-IPP algorithm are summarized in Algorithm~\ref{alg:mcIPP}. By \eqref{eq:sampleComp}, when $\alpha$ is close to $0$, the iterates move slowly, and the required sample size is small; whereas when
$\alpha$ is close to $1,$ the iterates move fast and the required sample size is large. 
In MC-IPP, we adaptively update the damping parameter $\alpha_k$ at each iteration $k$ based on whether a sufficient decrease in the function value is achieved
(see Line~\ref{line:alpha_min} and Line~\ref{line:alpha_max}). The algorithm also adaptively increases the sample size $N_k$ and decreases the parameter $\delta_k$.
When an increase in the function value is observed compared to several previous iterates, we reject the MC estimate with a positive probability $p$ (see Line~\ref{line:reject}) and resample. Additionally, MC-IPP can be warm-started by computing an MC estimate of \eqref{eq:initial} with $N_0$ uniformly distributed sample points.

\begin{algorithm}[htp!]
\caption{Monte-Carlo Inexact Proximal Point Method (MC-IPP)}\label{alg:mcIPP}
\begin{algorithmic}[1] 
    \STATE \textbf{Input:} $x^0\in\Re^d$, $0<\delta_0< 1$, $0<\eta_-<1<\eta_+$, $0<\theta_1\le\theta_2<1$, $\bar\epsilon>0$, 
    $0<\eta<1$, $T>0$, $0<\tau\le t_0\le T$, 
    $0<\alpha_{\min}\le\alpha_0\le\alpha_{\max}\le 1$, $0<p<1$, $N_0>0$, $C>1$, $0<c<1$, $m\ge 1$, $k_{\max}>0$, and $\epsilon_{stop}>0$
    \STATE $k \gets 0$
    \WHILE{$k < k_{\max}$}
        \STATE $y^{k} \gets \alpha_k \left(\prox_{t_k f}^{\delta_k}(x^k)\right)_{\MC} + (1-\alpha_k)x^{k}$ \label{line:yk}
        \IF{$k \ge m-1$ \textbf{and} $f(y^{k}) > \max\{f(x^k), f(x^{k-1}), \dots, f(x^{k-m+1})\} - \eta/k$}
            \IF{$f(y^{k}) \ge \max\{f(x^k), f(x^{k-1}), \dots, f(x^{k-m+1})\}$}
                \STATE Reject $y^k$ and return to Line~\ref{line:yk} with probability $p$ \label{line:reject}
            \ENDIF
            \STATE $x^{k+1} \gets y^k$
            \STATE $\delta_{k+1} \gets c \delta_k$
            \STATE $\alpha_{k+1} \gets \max\{\alpha_{\min}, c\alpha_k\}$ \label{line:alpha_min}
            \STATE $N_{k+1} \gets C N_k$
        \ELSE
            \STATE $x^{k+1} \gets y^k$
            \STATE $\delta_{k+1} \gets \delta_k$
            \STATE $\alpha_{k+1} \gets \min\{\alpha_k / c, \alpha_{\max}\}$ \label{line:alpha_max}
            \STATE $N_{k+1} \gets N_k$
        \ENDIF
        \STATE Determine $t_{k+1}$ using Lines~\ref{line:qk}--\ref{line:tk} of Algorithm~\ref{alg:IPP}
        \STATE $k \gets k + 1$
        \IF{$\norm{x^{k+1} - x^k} < \epsilon_{stop}$}
            \STATE \textbf{Break}
        \ENDIF
    \ENDWHILE
    \STATE \textbf{Output:} last iterate $x^k$
\end{algorithmic}
\end{algorithm}

The almost sure convergence of MC-IPP is a corollary of the convergence of IPP methods proved in Theorem~\ref{thm:sIPP}.
\begin{corollary}
For $f:\Re^d\to\Re$, suppose Assumptions~\ref{assum1} -- \ref{assum3} hold, and 
one of the conditions listed in Proposition~\ref{prop:single_prox} holds.
Let $\Set{x^k}_{k\ge 0}$ be the sequence of iterates generated by Algorithm~\ref{alg:mcIPP}. 
If $T>0$ and $C>1$ chosen for Algorithm~\ref{alg:mcIPP} are sufficiently large, and $\alpha_{\min}>1-\eta_-$, 
then $\Set{x^k}_{k\ge 0}$ converges to $x^*$ almost surely
as $k\to\infty.$
\end{corollary}
\begin{proof}
Assume that $\lim\limits_{k\to\infty}\delta_k > 0$ or $\lim\limits_{k\to\infty} N_k<\infty$. 
Following similar arguments as in the proof of Corollary~\ref{cor:ttIPP},
the right-hand side of \eqref{eq:fdecrease} would decrease to $-\infty$ as $k\to\infty$, which
contradicts the assumption that $f$ is bounded below. Therefore, $\lim\limits_{k\to\infty}\delta_k = 0$
and $\lim\limits_{k\to\infty} N_k=\infty.$
As shown in the proof of Theorem~\ref{thm:sIPP}, $t_k = T$ for all $k$ sufficiently large almost surely.
If any of the conditions listed in Proposition~\ref{prop:single_prox} holds for $f$ and $T>0$ is sufficiently large,
then the proximal operator $\prox_{t_k f}(x^k)$ is single-valued for all $k$ sufficiently large almost surely.
By Theorem~\ref{thm:conv} and standard results of MC integration, if choosing $C>1$ sufficiently large in Algorithm~\ref{alg:mcIPP},
the condition \eqref{eq:pprox_error} holds. Therefore, if $T>0$ and $C>1$ chosen for Algorithm~\ref{alg:mcIPP} are sufficiently large, and $\alpha_{\min}>1-\eta_-$,
by Theorem~\ref{thm:sIPP}, $\Set{x^k}_{k\ge 0}$ converges to $x^*$ almost surely
as $k\to\infty.$
\end{proof}

While a sufficiently large $C$ ensures the theoretical convergence of Algorithm~\ref{alg:mcIPP} according to the above corollary, 
a smaller $C$ reduces the required number of function evaluations. Therefore, in the numerical experiments presented in Section~\ref{sec:experiments}, 
the control parameters for MC-IPP are chosen heuristically.

We remark that although the proximal operator in \eqref{eqn:aprox} could, in principle, be approximated using other sampling algorithms—such as the Unadjusted Langevin 
Algorithm or Hamiltonian Monte Carlo—these methods do not offer significant improvement in the current setting. As illustrated in Fig.~\ref{fig:1d}, the function being sampled is approximately unimodal and relatively easy to handle. As a result, the simple Monte Carlo integration employed in Algorithm~\ref{alg:mcIPP} already yields satisfactory results. Extending our IPP framework to incorporate more advanced sampling methods that exploit gradient information remains a promising direction for future work.

\section{Experiments}\label{sec:experiments}
This section presents experimental results of TT-IPP and MC-IPP on a diverse set of benchmark functions and two practical applications.
More applications are presented in Appendix~\ref{append:applications}, including the use of TT-IPP for solving the Hamilton-Jacobi equation \cite{evans2022partial}.
\subsection{Experiments on benchmark functions}
\label{sec_Ex}
In this section, we test the proposed IPP algorithms on benchmark functions from the established function library \cite{benchmark_Andrea}, unless stated otherwise.
The performance of each algorithm in all numerical experiments is assessed using the accuracy metric \(\|x^k -x^*\|_{\infty}\), where \(x^*\) denotes the global minimizer, and \(x^k\) is the final iterate upon termination. For ease of comparison, test functions were shifted from their original definitions to ensure that their minimum values are within \([-1,1]^d\) and the minimum value is $0$. 
 We compare our algorithms with several existing global optimization algorithms discussed in Section~\ref{sec:prior}. 
For HJ-MAD and TT-Opt, we utilized the original implementations provided by the authors in \cite{heaton2024global, sozykin2022ttopt}. The implementation of CBO was based on the code from \cite{CBO_Analysis}. For PSO \cite{PSO_first}, PRS \cite{locatelli2013global}, and SA \cite{sa_original}, MATLAB's built-in functions were used, while the implementation of DE was from \cite{BuehrenDE_Code}. 
For particle-based methods including CBO, PSO, and DE, we used $40d$ particles in each iteration and reported the location of the best particle in the whole population. 

For the proposed TT-IPP in Algorithm \ref{alg:ttIPP} and MC-IPP in Algorithm \ref{alg:mcIPP}, choices of
the control parameters are summarized in Table~\ref{table:parms}. These parameters were selected heuristically, 
guided by theoretical error bounds or theoretical sample complexity derived in previous sections.
In general, we observed that our IPP algorithms exhibit robustness across different parameter choices.
For TT-IPP, the domain of the test functions is restricted on $\Omega=[-5,5]^d$ due to the requirement of constructing TT approximations, and a uniform mesh was used.
Our implementation of TT-IPP is based on the implementation of the TT-cross algorithm \cite{savostyanov2014quasioptimality} in the
TT toolbox \cite{TT-Toolbox}.  The initial guess $x^0$ was chosen to be a TT estimate of \eqref{eq:initial} obtained on a coarse initial mesh.
For MC-IPP, the initial sample size for estimating the proximal operator was set to be $N_0=40d$, where $d$ represents the dimension of each problem, and $x^0$ was chosen to be an MC estimate of \eqref{eq:initial} obtained using $40d$ initial sample points from the uniform distribution on $[-3,3]^d$.
For other iterative solvers, $x^0$ was chosen randomly from the uniform distribution on $[-3,3]^d$.

\begin{table}[htbp]
\centering
\begin{tabular}{l|p{0.39cm}p{0.39cm}p{0.39cm}p{0.39cm}p{0.39cm}p{0.39cm}p{0.39cm}p{0.39cm}p{0.39cm}p{0.39cm}p{0.39cm}p{0.39cm}}
\toprule
& $\delta_0$ & $\eta_-$ & $\eta_+$ & $\theta_1$ & $\theta_2$ & $\bar{\epsilon}$ & $\eta$ & $T$ & $\tau$ & $t_0$ &  $C$ & $m$ \\\midrule
\textbf{TT-IPP} & $0.1$ & $0.5$ & $2$ & $0.25$ & $0.75$ & $0.2$ & $10^{-3}$ & $20$ & $0.5$ & $1$ & $10^3$ & $4$ \\
\textbf{MC-IPP} & $0.1$ & $0.9$ & $2$ & $0.25$ & $0.75$ & $0.2$ & $10^{-3}$ & $20$ & $0.5$ & $1$ & $1.1$ & $4$ \\
\bottomrule
\end{tabular}
\caption{Control parameters for TT-IPP and MC-IPP algorithms. Moreover, $h_0 = 0.1$, $\gamma = 1.1$ for TT-IPP; $c = 0.9$, $\alpha_{min} = 0.2$, $\alpha_{max}=0.3$, $p = 0.8$ for MC-IPP} \label{table:parms}
\end{table}

Table \ref{table:TT_benchmark} compares TT-IPP with other solvers on benchmark functions under two scenarios: 
\begin{enumerate}
    \item  The number of function evaluations required to achieve the desired accuracy. 
    \item  The final error after a fixed number of function evaluations. 
\end{enumerate}
As shown in Table \ref{table:TT_benchmark}, TT-IPP significantly outperforms other methods, particularly in cases where $d\geq 20$. While the function evaluations for other methods increase almost exponentially with dimension, TT-IPP (and TT-Opt, to a lesser extent) demonstrate a nearly linear growth owing to the use of TT approximations.

\begin{table}[htbp]
\[\def\arraystretch{1.3}
\hspace*{-0.6cm}
\begin{array}{@{} l| *{5}{c} l |*{4}{c} @{}}
\toprule
\text{\makecell{Test\\ Problems}}     & \multicolumn{5}{c@{}}{\text{Func. Eval. till $\|x_k-x^*\|_{\infty}\leq 10^{-2}$}} & &\multicolumn{4}{c@{}}{\text{Avg. Error after 500 K Func. Eval.}} \\
\cmidrule(l){2-6} \cmidrule(l){8-11}
&\text{TT-IPP} &\text{DE} &\text{PSO} &\text{SA} &\text{TT-Opt} & &\text{TT-IPP} &\text{DE}
&\text{PSO} &\text{SA}\\
\cline{1-11} 
\text{Griewank } \mathbb{R}^4 & \textbf{5379} & 17K & -& - & 22K & &   \mathbf{3.31\times 10^{-4}}    & 5.06 &   1.25   &  1.62 \\
\text{Griewank } \mathbb{R}^{10}& \textbf{14K} & - & - &- &54K & & 
    \mathbf{5.15\times 10^{-5}}   & 9.80   & 2.65   & 1.31 \\
\text{Griewank } \mathbb{R}^{20} & \textbf{38K} &- & - &- &117K  & &     \mathbf{6.22\times 10^{-4}}    &6.37    & 2.13  &   2.13
\\
\text{Griewank } \mathbb{R}^{50} &   \textbf{69K} & - &- &- &309K &  &  \mathbf{2.93\times 10^{-4}} &    3.78    & 1.39  &   1.61 \\
\text{Griewank } \mathbb{R}^{100}   & \textbf{140K} & - & 476K& -&627K && \mathbf{2.97\times 10^{-4}} & 3.03& 1.90& 2.16\\
\cline{1-11} 
\text{Levy 3 } \mathbb{R}^{5}   & \textbf{3899} & 10K & 11K &- &22K& & \mathbf{3.37\times 10^{-5}}  &   9.00\times 10^{-5}  &   7.67\times 10^{-4}  & 2.25 \\
\text{Levy 3 } \mathbb{R}^{20}   & \textbf{17K} & 384K & 87K &- &117K &&   {4.16\times 10^{-4}} &   \mathbf{1.25\times 10^{-4}}   & 7.94\times 10^{-4}   & 2.81 \\
\text{Rastrigin } \mathbb{R}^{5}   & \textbf{3899} & 12K & 8200 &- &22K & &  \mathbf{6.89\times 10^{-4}} & 3.61\times 10^{-3} &    0.1002 &    2.68 \\
\text{Rastrigin } \mathbb{R}^{20}   & \textbf{17K} & - & - & -&117K & &  \mathbf{8.76\times 10^{-4}} &2.74   & 1.98   & 2.88 \\
\text{Ackley } \mathbb{R}^{5}   & {11K} & {12K} & \textbf{10K} & -&43K & & \mathbf{1.05\times 10^{-5}} & 3.61\times 10^{-3} &    7.23\times 10^{-4} &  2.35 \\
\text{Ackley } \mathbb{R}^{20}  & \textbf{112K} & 303K & 203K &- &235K & &   \mathbf{4.44\times 10^{-5}}  & 2.52\times 10^{-3}&   8.61\times 10^{-4}    &2.85\\
\text{Ackley } \mathbb{R}^{50}   & \textbf{336K} & - & - &- &618K & &  \mathbf{4.02\times 10^{-4}} & 3.22 & 9.48\times 10^{-4} &  2.95 \\
\cline{1-11} 
\text{Brwon } \mathbb{R}^{10}   & {7396} & 55K & 27K &\textbf{7152}&457K & &<10^{-16} & <10^{-16}  & 8.94\times 10^{-4} &7.34\times 10^{-4}  \\
\text{Exponential } \mathbb{R}^{10}   & \textbf{7494} & - & 
- & - &-& &\mathbf{<10^{-16}} &2.68   & 1.46  & 3.12\\
\text{Trid } \mathbb{R}^{10}   & \textbf{7434} & 47K & 34K &7591 &107K & &   {1.18\times 10^{-8}}  & \mathbf{ <10^{-16}} &  7.41\times 10^{-4} &    9.64\times 10^{-4} \\
\text{Schaffer 1 } \mathbb{R}^{5} \text{\cite{test_problems_2005}}  & \textbf{27K} & 292K & -&- &43K &&  \mathbf{6.46\times 10^{-8}} &2.81   & 2.96   & 2.88 \\
\text{Corrugated } \mathbb{R}^{5}   & \textbf{114K} & 322K & - & -&163K & &\mathbf{3.15\times 10^{-7}} &  0.431  & 0.515 &   2.445
 \\
\text{Cos. Mix. } \mathbb{R}^{20}  & \textbf{15K} & - &- &17K &235K& & \mathbf{6.43\times 10^{-6}}& 8.75 &  2.62   & 6.07\times 10^{-4} \\
\text{Alphine 1 } \mathbb{R}^{5}   & \textbf{26K} & 30K &- & -&41K & & \mathbf{4.72\times 10^{-4}} &    3.75\times 10^{-3} &  2.21   &3.14 \\
\bottomrule
\end{array}
\]
\caption{Comparison of TT-IPP with other algorithms. The values on the left represent the number of function evaluations required to achieve the desired accuracy $\|x-x^*\|_{\infty} \leq 10^{-2}$, with ``-" indicating that the method fails to converge after $1000K$ function evaluations. The values on the right represent final errors achieved either when $\|x-x^*\|_{\infty} \leq 10^{-3}$ or after $500K$  function evaluations. For the last three methods, the results are the average errors obtained from 50 random initial guesses.}
\label{table:TT_benchmark}
\end{table}

In Table~\ref{table:MC_benchmark}, we compare MC-IPP and other solvers when only a limited number of function evaluations are allowed. 
TT-IPP and TT-Opt are excluded from the comparison because they both rely on the computation of TT approximations on a mesh grid, reducing randomness in their results and limiting their ability to explore larger domains.
Results in Table~\ref{table:MC_benchmark} show that MC-IPP outperforms other methods in most cases, consistently demonstrating its advantage for functions defined on $\mathbb{R}^{10}$ and $\mathbb{R}^{20}$, 
making it a strong candidate for global optimization under resource constraints.

Figure~\ref{fig:path} illustrates the trajectories of different optimization algorithms. From Figure~\ref{fig:path}, it can be observed that TT-IPP provides a robust and direct convergence path to the global minimizer, leveraging the information across the entire domain efficiently. For MC-IPP, although the trajectory exhibits slight oscillations and requires more iterations, it maintains a direct and reliable path to the global minimizer without being trapped at local minimizers. In contrast, algorithms like DE and PSO tend to wander or converge to local minimizers within the domain.

Additionally, Table~\ref{tab:tt_comparison} compares the performance of TT-IPP in iteratively minimizing the Schaffer 02 function on \(\mathbb{R}^{10}\) with the direct evaluation of \eqref{eq:initial} for a small fixed \(\delta\). TT-IPP terminates when \(\|x^k - x^*\|_{\infty} \leq 10^{-5}\), while for the direct evaluation, the mesh size for the TT approximation is set as \(h = \delta / 2\). The first two rows of the table demonstrate that starting with a larger initial \(\delta\) enables TT-IPP to obtain a good initial guess on a coarser mesh, thereby reducing the number of function evaluations required for convergence.
On the other hand, using a smaller \(\delta\) leads to a TT approximation with a lower rank, as shown in Figure~\ref{fig:regula}, which reduces the associated storage requirements and computational costs. 
A comparison of the first two rows (TT-IPP results) with the last three rows (results from direct integral evaluation) demonstrates that TT-IPP achieves significantly higher accuracy with fewer function evaluations. This underscores the advantages of TT-IPP in delivering accurate solutions while ensuring computational efficiency.
 
\begin{table}[htbp]
\[\def\arraystretch{1.3}
\begin{array}{@{} l| *{7}{c} l}
\toprule
\text{\makecell{\\Optimization\\ Algorithm}} & \multicolumn{7}{c@{}}{\text{Error after $10^4$ Func. Eval. in $\mathbb{R}^{10}$}} &  \\
\cmidrule(l){2-8}
& \text{MC-IPP} & \text{HJ-MAD} & \text{CBO} & \text{PRS} & \text{DE} & \text{PSO} & \text{SA} & \\
\cline{1-8}
\text{Griewank} &  \mathbf{6.65 \times 10^{-2}} & 1.75 & >5 & 4.28 & >5 & >5 & >5 \\
\text{Levy 3 } & \mathbf{6.49 \times 10^{-1}} & 1.80 & 4.52 & 2.78 & 1.29 & 1.68 & 4.72 \\
\text{Zakharov} & \mathbf{1.71 \times 10^{-1}} & 1.81 & 3.14 & 2.56 & 4.06 & 3.12 \times 10^{-1} & 1.28 \\
\text{Ackley} & \mathbf{7.81 \times 10^{-2}} & 1.70 & 2.77 & 2.04 & 7.14 \times 10^{-1} & 8.93 \times 10^{-2} & 4.46 \\
\text{Rosenbrock}& \mathbf{2.70 \times 10^{-1}} & 2.06 & 3.97 & 2.28 & 1.20 & 1.00 & 1.57 \\
\text{Brown}& 1.08 \times 10^{-1} & 1.59 & >5 & 2.39 & 1.37 & \mathbf{1.02 \times 10^{-1}} & 1.00 \\
\text{Exponential}& \mathbf{7.01 \times 10^{-2}} & 1.89 & >5 & 4.35 & 4.50 & >5 & 4.63 \\
\text{Cos. Mix.} &\mathbf{6.75 \times 10^{-2}} & 1.82 & >5 & 2.48 & >5 & >5 & >5 \\
\text{Alphine 1}& \mathbf{1.81 \times 10^{-1}} & 2.54 & 4.79 & 3.32 & >5 & 4.17 & 3.85 \\
\text{Dropwave}& \mathbf{3.21 \times 10^{-1}} & 1.71 & 3.15 & 2.19 & 1.96 & 4.18 \times 10^{-1} & 4.31 \\
\toprule
& \multicolumn{7}{c@{}}{\text{Error after $4\times 10^4$ Func. Eval. in $\mathbb{R}^{20}$}} &  \\
\cmidrule(l){1-8}

\text{Griewank} & \mathbf{8.11 \times 10^{-2}} & 2.49 & >5 & 4.43 & >5 & 4.47 & 4.67 \\
\text{Rosenbrock} & \mathbf{5.97 \times 10^{-1}} & 2.16 & 4.49 & 3.23 & >5 & 1.40 & 1.08 \\
\text{Exponential} & \mathbf{1.15 \times 10^{-1}} & 2.88 & >5 & 4.69 & 4.77 & >5 & >5 \\
\text{Cos. Mix.} & \mathbf{5.85 \times 10^{-2}} & 2.38 & >5 & 4.22 & >5 & >5 & >5 \\
\text{Alphine 1} & \mathbf{1.33 \times 10^{-1}} & 2.92 & >5 & 3.64 & >5 & >5 & 4.31 \\
\text{Dropwave} & 5.07 \times 10^{-1} & {2.38} & 4.62 & 3.70 & 4.12 & \mathbf{4.69 \times 10^{-1}} & 4.85 \\
\toprule
\end{array}
\]
\caption{Comparison of MC-IPP with other algorithms. Average errors over $10$ runs are reported for benchmark functions on $\mathbb{R}^{10}$ and $\mathbb{R}^{20}$, with a maximum number of function evaluations.}
\label{table:MC_benchmark}
\end{table}

\begin{figure}
    \centering    
    \includegraphics[width=0.45\linewidth]{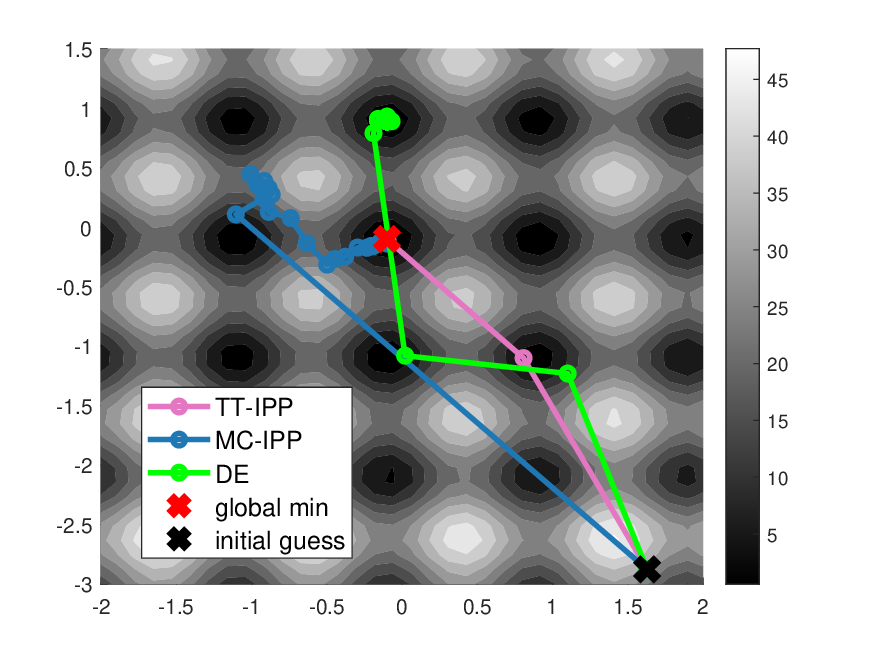}
    \includegraphics[width=0.45\linewidth]{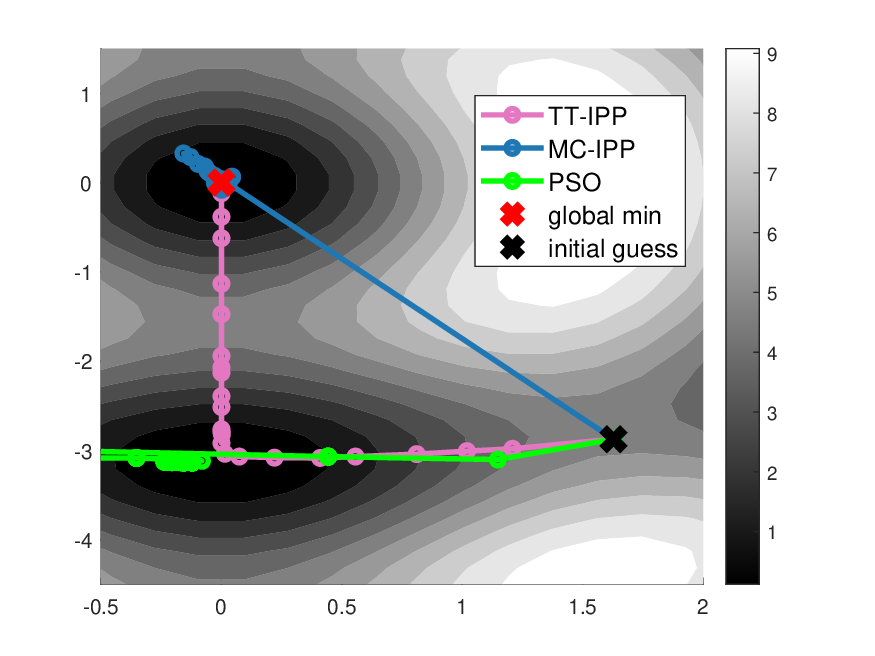}
    \caption{Trajectories of different algorithms for minimizing Rastrigin function (left) and Schaffer 1 function (right).}
    \label{fig:path}
\end{figure}

\begin{table}
\centering
\begin{tabular}{c|c|c|c|c|c}
\hline
& $\|x_k - x^*\|_{\infty}$ & \makecell{$\delta_K$} & \makecell{Initial \\ TT rank}& \makecell{Final \\ TT rank}& Func. Eval. \\ \hline
\makecell{TT-IPP \\with $\delta_0 = 0.5$} & $2.89 \times 10^{-8}$ & $3.91\times 10^{-3}$ & 7 & 1  &59K \\\hline
\makecell{TT-IPP \\with $\delta_0 = 0.1$} & $2.21\times 10^{-7}$ &$1.6\times 10^{-3}$ & 4 & 1 & 101K \\\hline
\makecell{Evaluating \eqref{eq:initial} \\with $\delta = 0.05$} & $5.82 \times 10^{-2}$ &- & 4&  -& 218K \\\hline
\makecell{Evaluating \eqref{eq:initial} \\with $\delta = 0.01$} & $1.87 \times 10^{-3}$ & - & 2 &-& 514K  \\ \hline
\end{tabular}
\caption{Comparison of TT-IPP using different initial $\delta_0$ and directly evaluating \eqref{eq:initial} for Schaffer 02 function with $d = 10$, where $\delta_K$ denotes the value of $\delta$ at termination.}\label{tab:tt_comparison}
\end{table}

\subsection{Practical applications}
We test our algorithms using two practical optimization problems from engineering applications. 

\begin{enumerate}
    \item 
    The first example is from \cite{soley2021iterative}, which involves a black-box optimization problem for identifying the global minimum energy configuration in a model of a DNA chain consisting of \( d = 50 \) hydrogen-bonded adenine-thymine (A-T) base pairs. The model calculates the total energy (in electronvolts) as a sum over all base pairs, where the energy of each base pair depends on the proton's position \( x_i \). The global minimum at \( x_i^* = -1 \) corresponds to the stable A-T configuration, while \( x_i = 1 \) represents the less stable tautomeric A*-T* configuration. This energy landscape results in a potential energy surface with \( 2^{50} \) local minima, reflecting the vast number of possible protonation states. Determining the minimum energy configuration is biologically significant, as abnormal hydrogen bonding can disrupt correct base pairing during DNA replication, a process linked to genetic mutations and cancer formation.

\item
The second example is from a financial application described in \cite{roncalli2013introduction}, where the goal is to optimize a portfolio such that each equity contributes equally to the overall risk. The objective function is defined as  
\[
f(w) = \sum_{i=1}^d\left(w_i(\Sigma w)_i - \frac{\sqrt{w^T\Sigma w}}{d}\right)^2,
\]
where $d = 10$, \(w_i\) represents the weight of each portfolio component, and \(\Sigma\) is the covariance matrix of returns. This optimization problem is nonconvex due to the square root term involving the variance. To address the constraint \(\sum w_i = 1\), a penalty term is added to the objective function. Additionally, since portfolio weights must be non-negative and bounded by 1, the search domain is restricted to \([0,1]^d\). For the covariance matrix \(\Sigma\), we pick $\Sigma_{ij} = \exp(-|i-j|^2/4)$. 
\end{enumerate}

 \begin{table}
\centering
\[
\begin{array}{c|c|c|c|c|c|c|c|c|c}
\hline 
\text{Problems} & \text{TT-IPP} & \text{MC-IPP} & \text{HJ-MAD} & \text{CBO} & \text{PRS} & \text{DE} & \text{PSO} & \text{SA} & \text{TT-Opt} \\\hline
\text{1} & 2.11 \times 10^{-15} & 1.56 & 2.34 & 1.95 & 1.27 & 4.36 & 2.34 & 2.00 & 5 \\
\text{2} & 7.92 \times 10^{-3} & 4.87 \times 10^{-2} & 0.550 & 0.981 & 0.236 & 0.784 & 0.599 & 0.397 & 9.21 \times 10^{-2} \\
\hline
\end{array}
\]
\caption{Error of methods for solving two practical problems using up to $100K$ function evaluations.}
\label{tab:practical}
\end{table}

Table~\ref{tab:practical} reports the error $\norm{x^k-x^*}_\infty$ at termination for different methods used to solve the two practical problems with a limit of $100K$ function evaluations,
highlighting the effectiveness of TT-IPP and MC-IPP.

\section{Conclusions}\label{sec:conclud}
In this work, we formulate a theoretical framework for inexact proximal point (IPP) methods for the global optimization of continuous nonconvex functions, 
establishing convergence guarantees under mild assumptions when either deterministic or stochastic estimates of proximal operators are used. The convergence of the expectation under the associated Gibbs measure as $\delta\to 0^+$ is established, and the convergence rate of $\bigO(\delta)$ is derived 
under additional assumptions. These results serve as a theoretical foundation for evaluating proximal operators inexactly using sampling-based methods
such as MC integration. Additionally, we introduce a new TT-based approach, accompanied by an analysis of the estimation error.
Furthermore, we propose two practical IPP algorithms. TT-IPP leverages TT estimates of the proximal operators, while
MC-IPP employs MC integration to estimate the proximal operators.
Both algorithms are designed to adaptively balance efficiency and accuracy in inexact evaluations of proximal operators. 
The effectiveness of the two algorithms is demonstrated through experiments on a diverse set of benchmark functions and various applications. 

The two IPP algorithms each have their advantages and limitations. While traditional global optimization methods typically suffer from the curse of dimensionality, TT-IPP employs the randomized TT cross algorithm and leverages the
Sobolev smoothness of functions to circumvent it, making it suitable for higher-dimensional problems. However, constructing a TT approximation over a mesh 
grid involves higher initial costs and restricts the search space of TT-IPP to a bounded domain, limiting its applicability to functions defined on larger or unbounded domains. On the other hand, MC-IPP benefits from easy implementation and is not restricted to bounded domains, though it often requires many function evaluations to achieve high accuracy.
Future work includes exploring other variance reduction techniques and rejection sampling \cite{gomes2023derivativefree} to 
enhance the performance of MC-IPP, developing strategies to integrate the strengths of TT and MC techniques for improved efficiency, training machine learning models to approximate proximal 
operators \cite{cassioli2012machine}, 
and extending the algorithms to optimization problems with constraints and noise.

 \section*{Declarations}
\textbf{Funding}~
M. Zhang and H. Schaeffer were supported in part by NSF 2331033 and NSF 2427558.
F. Han and S. Osher were partially supported by AFOSR MURI FA9550-18-502 and ONR N00014-20-1-2787. 
Y. Chow was supported in part by NSF DMS-2409903 and ONR N000142412661.
\vspace{0.5em}

\noindent\textbf{Competing interests}~
The authors have no competing interests to declare that are relevant to the content of this article.
\vspace{0.5em}

\noindent\textbf{Data and Code availability}~
The data that support the findings of this study are available upon reasonable request.
The code is available at \url{https://github.com/fq-han/ipp-global-opt}.




\bibliographystyle{spmpsci}      
\bibliography{ref}

\appendix
\section{Additional proofs}\label{append:proofs}
\subsection{Proof of Theorem~\ref{thm:sIPP}}\label{app:proof27}
First, we show that $\lim\limits_{k\to\infty}t_k = T$ almost surely.
By \eqref{eq:pprox_error}
$\sum_{k=0}^\infty\pr{\norm{y^k-\hat x^k}^2>\epsilon_k}<\infty.$
The Borel–Cantelli lemma implies that
$\pr{\norm{y^k-\hat x^k}^2>\epsilon_k ~\textrm{ infinitely often}}=0$,
i.e. \[
\pr{\norm{y^k-\hat x^k}^2\le \epsilon_k ~\textrm{ for all sufficiently large } k}=1\,.
\]
Therefore, 
$\pr{\sum\limits_{k=1}^\infty \norm{y^k-\hat x^k}^2<\infty}=1.$
Then by \eqref{eq:boundfk},
\begin{equation}\label{eq:xk}
\pr{\sum_{k=1}^\infty \norm{x^{k}-\hat x^k}^2<\infty}=1\,.
\end{equation}
By the triangle inequality, 
$\pr{\lim\limits_{k\to\infty} \norm{x^{k+1}-x^k}^2=0}=1$,
which, by Lines~\ref{line:qk}--\ref{line:tk} of Algorithm~\ref{alg:IPP}, implies that $\lim\limits_{k\to\infty}t_k = T$ almost surely.

Next, we show that $\Set{\norm{\hat x^k}^2, \norm{x^k}^2}_{k\ge 0}$ is uniformly bounded by some constant $M$ with probability $1$. 
Indeed, \eqref{eq:boundfk} implies that $\Set{f(\hat x^{k})}$ is uniformly bounded by some constant $M'$ with probability $1$.
Hence, by
Assumption~\ref{assum2}, $\Set{\hat x^k}$ is uniformly bounded with probability $1$. Combining with \eqref{eq:prox_error}, it follows that
$\Set{x^k}_{k\ge 0}$ is  uniformly bounded with probability $1$ as well.

By Bolzano-Weiserstrass theorem, with probability $1$, there exists a subsequence $\Set{\hat x^{k_j}}$ that converges to some
limit $x^{\infty}.$ Following the same arguments in the proof of Theorem~\ref{thm:IPP}, if
\begin{equation}\label{eq:ptbound} T>M/\mu\,, \end{equation}
then $x^{\infty}=x^*$ with probability $1$. Since every subsequence of $\Set{\hat x^{k}}$ converges to $x^*$ almost surely, 
$\Set{\hat x^{k}}$ itself converges to $x^*$ almost surely when \eqref{eq:ptbound} holds.
Therefore, by \eqref{eq:xk}, we conclude that $\Set{x^{k}}$ itself converges to $x^*$ almost surely when \eqref{eq:ptbound} holds.

\subsection{Proof of Proposition~\ref{prop:single_prox}}\label{app:proof31}
For part (1), if $f$ is prox-regular, then by \cite[Thm. 1.3]{poliquin2010calculus}, $\operatorname{prox}_{tf}(x)$ is single-valued for all $t<1/r.$

For part (2), if  $f$ is twice continuously differentiable on a neighborhood $U$ of $x^\ast$  and $x^\ast$ is nondegenerate,
then there exists a subset $\tilde U\subset U$ such that $\xstar\in\tilde U$ and 
\begin{equation}\label{eq:epsilonH}
\norm{\Grad f(z_1)-\Grad f(z_2)} \ge \min_{z\in\tilde U}\lambda_{\min} \left (\Hess f(z)\right)\norm{z_1-z_2}:=
\epsilon_H\norm{z_1-z_2}, \qquad \forall ~z_1, z_2\in \tilde U\,,
\end{equation}
for some constant $\epsilon_H>0$. Now assume that $z_1, z_2\in \prox_{tf}(x)$, then
\[
f(z_i)+\frac{1}{2t}\norm{z_i-x}^2\le f(\xstar)+\frac{1}{2t}\norm{\xstar-x}^2
\]
for $i=1,2,$ which implies that
\[
0\le f(z_i)-f(\xstar)\le \frac{1}{2t}(\norm{\xstar-x}^2-\norm{z_i-x}^2)
\]
for $i=1,2.$ Since $\xstar$ is the unique minimizer and $f$ is p-coercive, $z_1,z_2\in\tilde U$ if $t$ is sufficently large.
Moreover, by the definition of $\prox_{tf}(x),$ 
$\Grad f(z_i) = (x-z_i)/t$ for $i=1,2\,.$
It follows that 
\[
\norm{\Grad f(z_1)-\Grad f(z_2)}=\frac{1}{t}\norm{z_1-z_2}<\epsilon_H\norm{z_1-z_2}
\]
when $t>1/\epsilon_H$. Hence, combining with \eqref{eq:epsilonH}, we have $z_1 = z_2,$ i.e. $\operatorname{prox}_{t f}(x)$ is single-valued.

Finally, for part (3), assume that $f$ is sharp on a neighborhood $U$ of $x^\ast.$ Let $\tilde U\subset U$ be a compact neighborhood of $\xstar.$
Then for sufficiently large $t$, 
\[
f(\xstar)+\frac{1}{2t}\norm{z-\xstar}<f(z), \qquad \forall ~z\in \Re^d\setminus \tilde U\,,
\]
which implies $\prox_{tf}(x)\subset \tilde U.$
Since $\tilde U$ is compact, there exists a constant $M>0$ such that, for arbitrary $z\in\tilde U,$
\[
\abs{\norm{\xstar-x}^2-\norm{z-x}^2} = \abs{\langle \xstar-z, \xstar+z-2x\rangle}\le M\norm{\xstar-z}\,.
\]
On the other hand, by \eqref{eq:sharp}, for arbitrary $z\in\tilde U$ and $t> M/\eta,$
\[
f(z)-f(\xstar) \ge \eta \norm{z-\xstar}> \frac{M}{t}\norm{z-\xstar}\ge \norm{\xstar-x}^2-\norm{z-x}^2\,,
\]
which implies 
\[
f(z)+\frac{1}{2t}\norm{z-x} > f(\xstar)+\frac{1}{2t}\norm{\xstar-x} , \qquad \forall ~z\in \tilde U\,.
\]
Hence, for $t$ sufficiently large, $\prox_{tf}(x)$ is single-valued and $\prox_{tf}(x)=\xstar.$

\subsection{Proof of Theorem~\ref{thm:conv}}\label{app:proof34}
For $\delta>0,$ write
\begin{equation*}
z_\delta := \frac{\int z\exp\left(-\phi(z)/\delta\right)dz}{\int\exp{(-\phi(z)/\delta)dz}}\,.
\end{equation*}
It follows that
\[
z_\delta-z^* = \frac{\int (z-z^*)\exp\left(-\phi(z)/\delta\right)dz}{\int\exp{(-\phi(z)/\delta)dz}}\,.
\]
Define $\tilde \phi(z):=\phi(z)-\phi(z^*).$ Then $\tilde \phi\ge 0$ and 
\[
z_\delta-z^* = 
\frac{\int (z-z^*)\exp{(-\tilde \phi(z)/\delta)}dz}{\int\exp{(-\tilde \phi(z)/\delta)dz}}\,.
\]
As $\tilde \phi$ is also $p$-coercive, there exists some constant $\eta>0$ such that $\norm{z-z^*}\ge\eta$
implies that $\tilde \phi(z)\ge\norm{z-z^*}^p.$ 
It follows that, if  $\norm{z-z^*}\ge\eta$, then $\tilde \phi(z)\ge c\eta^p+(1-c)\norm{z-z^*}^p$ for any $c\in (0,1).$
Hence, as $\delta\to 0^+,$
\[
\norm{\int_{\Set{z:\norm{z-z^*}\ge\eta}} (z-z^*)\exp{\left(-\frac{\tilde \phi(z)}{\delta}\right)}dz} = o\left(\exp\left(-\frac{\eta^p}{2\delta}\right)\right)\,,
\]
and
\[
\int_{\Set{z:\norm{z-z^*}\ge\eta}}\exp{\left(-\frac{\tilde \phi(z)}{\delta}\right)}dz =  o\left(\exp\left(-\frac{\eta^p}{2\delta}\right)\right)\,.
\]

As there exists a neighborhood $U$ of $z^*$
such that $\phi$ restricted on $U$ is continuous, for $\epsilon>0$
sufficiently small, $U_{\epsilon} := \Set{z:\norm{z-z^*}\le\epsilon}\subset U$ and 
$\min\limits_{\epsilon \le \norm{z-z^*}\le\eta}\tilde \phi(z) = \tilde\epsilon$
for some $\tilde\epsilon>0.$ Therefore, 
\[
\norm{\int_{\Set{z:\epsilon \le\norm{z-z^*}\le\eta}} (z-z^*)\exp{\left(-\frac{\tilde \phi(z)}{\delta}\right)}dz} =  o\left(\exp\left(-\frac{\tilde\epsilon}{2\delta}\right)\right) \,,
\]
and
\[
\int_{\Set{z:\epsilon \le\norm{z-z^*}\le\eta}} \exp{\left(-\frac{\tilde \phi(z)}{\delta}\right)}dz =o\left(\exp\left(-\frac{\tilde\epsilon}{2\delta}\right)\right) \,.
\]
Moreover, as $\tilde \phi(z^*)=0$ and $\tilde \phi$ is continuous on $U_\epsilon\subset U$,
\[
\int_{\Set{z:\norm{z-z^*}<\epsilon}} \exp{\left(-\frac{\tilde \phi(z)}{\delta}\right)}dz=o\left(\exp\left(-\frac{\min\{\eta^p,\tilde\epsilon\}}{2\delta}\right)\right)\,.
\]
Therefore,
\begin{align*}
\lim_{\delta\to 0^+}\norm{z_\delta-z^*} =& \lim_{\delta\to 0^+} \norm{\frac{\int_{\Re^d} (z-z^*)\exp\left(-\phi(z)/\delta\right)dz}{\int\exp{(-\phi(z)/\delta)dz}}}\\
=& \lim_{\delta\to 0^+} \norm{\frac{\int_{\Set{z:\norm{z-z^*}<\epsilon}} (z-z^*)\exp\left(-\phi(z)/\delta\right)dz}{\int_{\Set{z:\norm{z-z^*}<\epsilon}}\exp{(-\phi(z)/\delta)dz}}}
\le \epsilon\,.
\end{align*}
Since $\epsilon >0$ can be arbitrarily small,
$\lim\limits_{\delta\to 0^+}\norm{z_\delta-z^*} = 0$, which completes the proof. 

\subsection{Proof of Theorem~\ref{thm_C2}}\label{app:proof35}
We first include two useful lemmas.
\begin{lemma}[Morse lemma \cite{gamkrelidze2012analysis}]\label{lem:morse}
Let $z^*$ be a nondegenerate stationary point of $\phi$. There exists an open neighborhood $U$ of $z^*$ and a 
homeomorphism $T:U\to V \subset\Re^d$ such that
$T(z^*) = 0$, $\det(T'(z^*))=1$,
and, with $y=T(z),$
$\phi(z) = \sum_{i=1}^d \lambda_i y_i^2$,
where $\lambda_1,\lambda_2,\cdots,\lambda_n$ are eigenvalues of $\Hess \phi(z^*)$.
\end{lemma}

In particular, if $\phi$ is $C^2$ around the stationary point, then
a twice continuously differentiable homeomorphism $T$ can be obtained by introducing hyperspherical coordinates (see e.g., \cite[Section 3.7]{miller2006applied}).

\begin{lemma}[{\cite[Chapter 2]{miller2006applied}}]\label{lem:integral}
Let $\alpha$ and $\beta$ be positive numbers. For any function $\psi\in C^2(\Re),$ as $\lambda\to+\infty$,
\[
\int_{-\alpha}^\beta \psi(y)\exp{(-\lambda y^2)}dy = \sqrt{\frac{\pi}{\lambda}}\left(\psi(0)+\frac{\psi''(0)}{2}\lambda^{-1}+\mathcal O(\lambda^{-2})\right).
\]
\end{lemma}

Next, we are ready to prove Theorem~\ref{thm_C2}.

\noindent\textbf{Proof of Theorem~\ref{thm_C2}.}
For $\delta>0,$ write
\[
z_\delta := \frac{\int z\exp\left(-\phi(z)/\delta\right)dz}{\int\exp{(-\phi(z)/\delta)dz}}.
\]
It follows that
\[
z_\delta-z^* = \frac{\int(z-z^*)\exp\left(-{\phi(z)}/{\delta}\right)dz}{\int\exp{(-\phi(z)/\delta)dz}}\,.
\]
Define $\tilde \phi(z):=\phi(z)-\phi(z^*).$ Then $\tilde \phi\ge 0$ and 
\[
z_\delta-z^*
= \frac{\int (z-z^*)\exp{(-\tilde \phi(z)/\delta)}dz}{\int \exp{(-\tilde \phi(z)/\delta)dz}}.
\]

As $\tilde \phi$ is also $p$-coercive, there exists some constant $\eta>0$ such that $\norm{z-z^*}\ge\eta$
implies that $\tilde \phi(z)\ge\norm{z-z^*}^p.$ Hence, as $\delta\to 0^+,$
\[
\norm{\int_{\Set{z:\norm{z-z^*}\ge\eta}} (z-z_\delta)\exp{\left(-\frac{\tilde \phi(z)}{\delta}\right)}dz} = o\left(\exp\left(-\frac{\eta^p}{2\delta}\right)\right)\,,
\]
and
\[
\int_{\Set{z:\norm{z-z^*}\ge\eta}} \exp{\left(-\frac{\tilde \phi(z)}{\delta}\right)}dz = o\left(\exp\left(-\frac{\eta^p}{2\delta}\right)\right)\,.
\]
By Lemma~\ref{lem:morse}, there exists an open neighborhood $V$ of $z^*$ and a homeomorphism $T$ such that $T(z^*)=0,$ $\det(T'(z^*))=1$, and
$\tilde \phi(z) = \frac{1}{2}\sum_{i=1}^d\lambda_iy_i^2\,,$
where $y=T(z)$ and $\lambda_1,\cdots,\lambda_n$ are eigenvalues of $\Hess \phi(z^*)$. Consider $\tau>0$ sufficiently small such that
\[
C_\tau :=\Set{y:\norm{y}_\infty\le\tau}\subset T(U\cap V)\,.
\]
It follows that
\begin{align}\label{eqn:integral}
&\int_{T^{-1}(C_\tau)} z\exp{\left(-\frac{\tilde \phi(z)}{\delta}\right)}dz \nonumber
=\int_{C_\tau}T^{-1}(y)\exp\left(-\frac{\sum_{i=1}^d \lambda_i y_i^2}{2\delta}\right)\det((T^{-1})'(y))dy \nonumber\\
=&\int_{C_\tau}(z^*+\mathcal O(y))exp\left(-\frac{\sum_{i=1}^d \lambda_i y_i^2}{2\delta}\right)dy \nonumber\\
=&z^*\prod_{i=1}^d\int_{-\tau}^{\tau}\exp\left(-\frac{\lambda_i y_i^2}{2\delta}\right)dy_i+\int_{C_\tau}\mathcal O(y)\exp\left(-\frac{\sum_{i=1}^d \lambda_i y_i^2}{2\delta}\right)dy \nonumber\\
=&\frac{(2\pi\delta)^{d/2}}{\sqrt{\det(\Hess\tilde f(z^*))}}z^*+\int_{C_\tau}\mathcal O(y)\exp\left(-\frac{\sum_{i=1}^d \lambda_i y_i^2}{2\delta}\right)dy\,.
\end{align}
Hence, by Lemma~\ref{lem:integral}, as $\delta\to 0^+,$
\begin{align*}
\norm{\int_{T^{-1}(C_\tau)} z\exp{\left(-\frac{\tilde \phi(z)}{\delta}\right)}dz-\frac{(2\pi\delta)^{d/2}}{\sqrt{\det(\Hess\tilde \phi(z^*))}}z^*}=\bigO(\delta^{d/2+1})\,.
\end{align*}
Similarly, as $\delta\to 0^+,$
\[
\int_{T^{-1}(C_\tau)}\exp{\left(-\frac{\tilde \phi(z)}{\delta}\right)}dz = \frac{(2\pi\delta)^{d/2}}{\sqrt{\det(\Hess\tilde \phi(z^*))}}+\bigO
(\delta^{d/2+1})\,.
\]
For $z\in\Set{z:\norm{z-z^*}\le\eta}\setminus T^{-1}(C_\tau)$, since $z^*$ is the unique minimizer and $T^{-1}(C_\tau)$ is an open neighborhood of $z^*,$
there must exist $\epsilon>0$ such that
$\tilde \phi(z)\ge\epsilon$. Hence, as $\delta\to 0^+,$
\[
\norm{\int_{\Set{z:\norm{z-z^*}\le\eta}\setminus T^{-1}(C_\tau)}(z-z^*)\exp{\left(-\frac{\tilde \phi(z)}{\delta}\right)}dz}=o(e^{-\epsilon/2\delta})\,,
\]
and
\[
\int_{\Set{z:\norm{z-z^*}\le\eta}\setminus T^{-1}(C_\tau)}\exp{\left(-\frac{\tilde \phi(z)}{\delta}\right)}dz=o(e^{-\epsilon/2\delta})\,.
\]

Therefore,\[
\lim_{\delta\to 0^+} z_\delta = \lim_{\delta\to 0^+}\frac{\frac{(2\pi\delta)^{d/2}}{\sqrt{\det(\Hess\tilde \phi(z^*))}}z^*+\bigO(\delta^{d/2+1})}{\frac{(2\pi\delta)^{d/2}}{\sqrt{\det(\Hess\tilde \phi(z^*))}}+\bigO(\delta^{d/2+1})} = z^*\,.
\]
Moreover,
\[
\lim_{\delta\to 0^+} \norm{z_\delta-z^*}= \frac{\bigO(\delta^{d/2+1})}
{\frac{(2\pi\delta)^{d/2}}{\sqrt{\det(\Hess\tilde \phi(z^*))}}+\bigO(\delta^{d/2+1})} = \bigO(\delta)\,.
\]

\subsection{Proof of Corollary~\ref{cor:convexHull}}\label{app:proof36}
Let \[
\bar z^*= \left(\sum\limits_{j-1}^m \frac{z_j^*}{\sqrt{\det(\Hess \phi(z_j^*))}} \right)/
\left(\sum\limits_{j-1}^m \frac{1}{\sqrt{\det(\Hess\phi(z_j^*))}} \right).
\]
Following similar arguments as in the proof of Theorem~\ref{thm_C2}, 
\[
\lim_{\delta\to 0^+} \norm{z_\delta-\bar z^*}= \frac{\bigO(\delta^{d/2+1})}
{\sum\limits_{j-1}^m \frac{(2\pi\delta)^{d/2}}{\sqrt{\det(\Hess\phi(z_j^*))}}+\bigO(\delta^{d/2+1})} = \bigO(\delta).
\]

\section{Other applications: solving Hamilton-Jacobi equation}\label{append:applications}
In this section, we explore the use of TT-IPP for solving the Hamilton-Jacobi (HJ) equation.
We aim to solve the following HJ equation:
\[
\begin{cases}
    \frac{\partial u}{\partial t} + H(\nabla u) = 0, & t > 0\,, \\
    u(x, 0) = f(x), & t = 0\,.
\end{cases}
\]
According to the Hopf-Lax formula, when \( H \) is convex, the solution is given by
\[
u(x, t) = \min_y \left\{ f(y) + t H^*\left( \frac{x - y}{t} \right) \right\}, \quad t > 0\,.
\]
We apply our proposed TT-IPP algorithm to solve this optimization problem, obtaining an approximation \(\tilde{y}\) to the global minimum \(y\) and constructing an approximate solution \(\tilde{u}\) as
\begin{equation}
\label{appx_inf_conv_HJ}
   \tilde{u}(x, t) = f\left(\tilde{y}(x, t)\right) + t H^*\left(\frac{\tilde{y}(x, t) - x}{t}\right)\,.
\end{equation}
To measure the accuracy of the solution, we introduce the residual function
\[
r(x, t) := \left| \frac{\partial \tilde{u}(x, t)}{\partial t} + H(\nabla  \tilde{u}(x, t)) \right|\,.
\]

We investigate the accuracy of the approximation \eqref{appx_inf_conv_HJ} for different convex Hamiltonians given by
\[
H(x) = \frac{\|x\|_p^p}{p}, \quad H^*(x) = \frac{\|x\|_q^q}{q}, \quad \frac{1}{p} + \frac{1}{q} = 1\,.
\]
The \( L^2 \)-norm of the residual function is computed for various values of \( \delta \) and dimensions. We fix \( t = 1 \) and evaluate the residual function at 100 randomly sampled points in \( x \in [-2, 2]^d \), with two different nonconvex initial conditions $f_1(x) = \|x\|_{1/2}^{1/2}$ and $f_2(x) = \sum_{i=1}^d(\sin(\pi x_i)+1)$. The results are summarized in Table~\ref{table_HJ}, demonstrating that our mesh-free approximation can provide a reasonably accurate approximation to the solution of the original HJ equation. 
Figure~\ref{fig:hj_solution} presents contour plots of the 2D slice of the approximate solution to the HJ equation, with initial data \( \|x\|_{1/2}^{1/2} \) and Hamiltonian \( H(u) = \|u\|_2^2 / 2 \), evaluated at \( t = 0 \), \( t = 0.2 \), and \( t = 2 \) in a 10-dimensional space. 

\begin{table}[h!]
\centering
\begin{tabular}{c|c|c|c|c|c|c}
\hline
\( \mathbf{d} \) & \(\makecell{p=2\\f= f_1}\) & \(\makecell{p=4\\f= f_1}\) & \(\makecell{p=6\\f= f_1}\) & \(\makecell{p=2\\f= f_2}\) & \(\makecell{p=4\\f= f_2}\) & \(\makecell{p=6\\f= f_2}\) \\
\hline
32  & \(1.18\times 10^{-3} \) & \( 9.22 \times 10^{-4} \) & \(1.07\times 10^{-3} \) & \(  2.16\times 10^{-3} \) & \(1.57  \times 10^{-3} \) & \( 1.24 \times 10^{-3} \) \\
64  & \(2.48 \times 10^{-3} \) & \( 1.31 \times 10^{-3} \) & \( 2.97 \times 10^{-3} \) & \( 4.03 \times 10^{-3} \) & \(2.15 \times 10^{-3} \) & \( 1.75 \times 10^{-3} \) \\
128 & \(3.51 \times 10^{-3} \) & \( 2.75 \times 10^{-3} \) & \( 3.43 \times 10^{-3} \) & \( 6.67\times 10^{-3} \) & \( 3.54 \times 10^{-3} \) & \(2.81 \times 10^{-3} \) \\
\hline
\end{tabular}
 \label{table_HJ}
 \caption{$L^2$-norm of residual function for different \( d \), \(p \), and initial conditions $f_1(x) = \|x\|_{1/2}^{1/2}$, $f_2(x) = \sum_{i=1}^d(\sin(\pi x_i)+1)$.}
\end{table}
\begin{figure}
    \centering
    \includegraphics[width=0.32\linewidth]{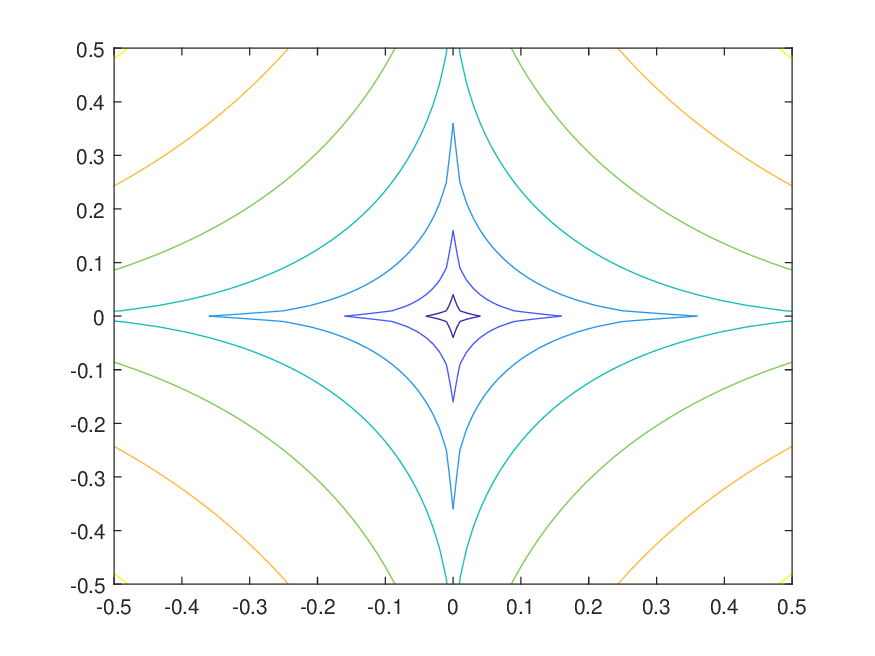}
    \includegraphics[width=0.32\linewidth]{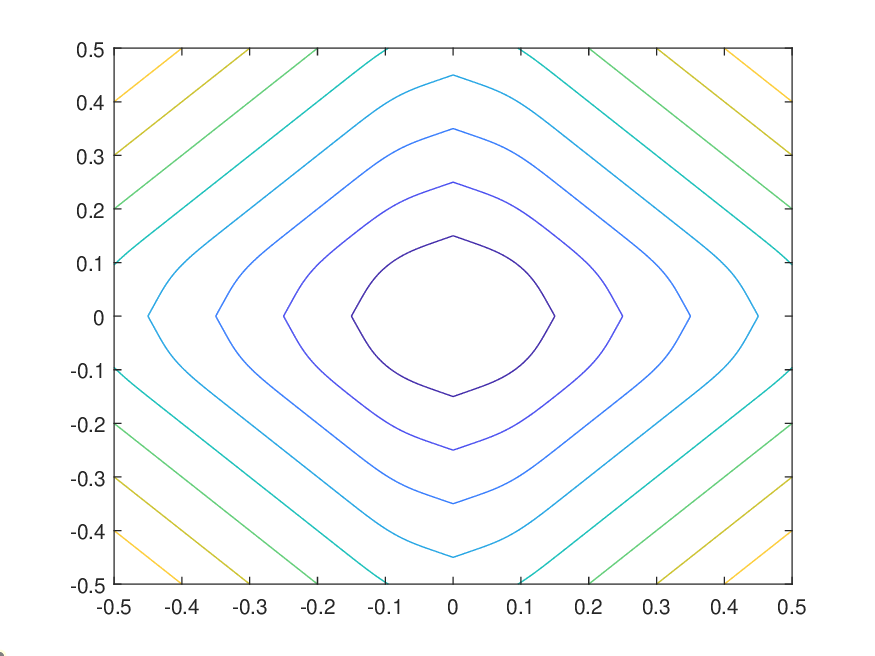}
    \includegraphics[width=0.32\linewidth]{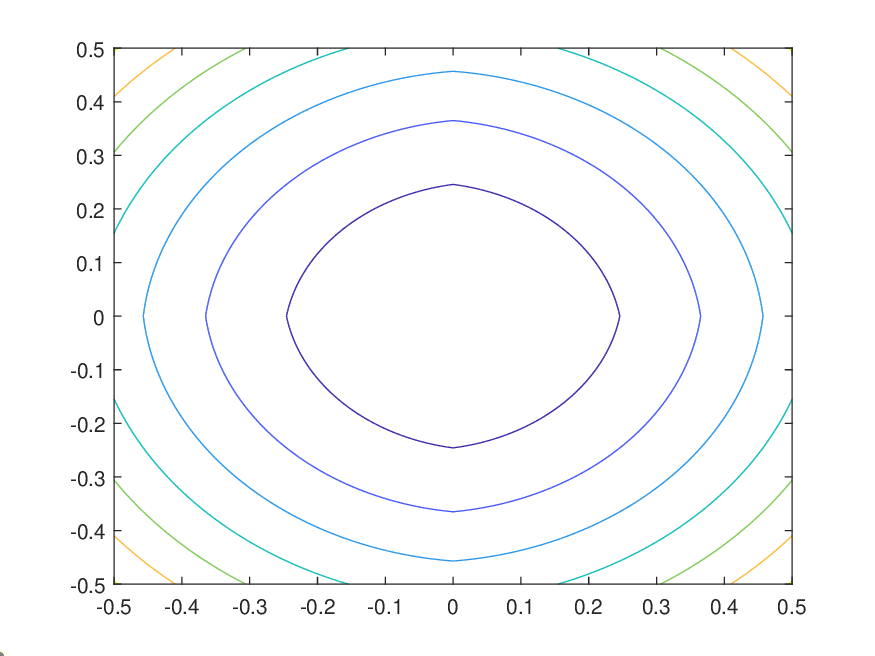}
    \caption{Contour plots of the Hamilton-Jacobi equation solution with initial condition \( \|x\|_{1/2} \) and Hamiltonian \( H(u) = \|u\|_2^2 \) at \( t = 0 \), \( t = 0.2 \), and \( t = 2 \) in a 10-dimensional space.}
    \label{fig:hj_solution}
\end{figure}
\end{document}